\documentclass[a4paper, reqno, 10pt]{amsart}

\usepackage{comment}
\usepackage{mathtools}
\usepackage{amsmath}
\usepackage{etoolbox}
\usepackage[english]{babel}
\usepackage{amssymb,amsmath,amsthm}
\usepackage{bbm}
\usepackage{fancyhdr}
\usepackage{ifthen}
\usepackage{color}
\provideboolean{shownotes} 
\setboolean{shownotes}{true}
\usepackage{hyperref}
\usepackage{graphicx}
\newcommand{\mres}{\mathbin{\vrule height 1.6ex depth 0pt width
		0.13ex\vrule height 0.13ex depth 0pt width 1.3ex}}

\newcommand{\margnote}[1]{
	\ifthenelse{\boolean{shownotes}}%
	{\marginpar{\raggedright\tiny\texttt{#1}}}%
	{}%
}

\newcommand{\hole}[1]{
	\ifthenelse{\boolean{shownotes}}%
	{\begin{center} \fbox{ \rule {.25cm}{0cm}
				\rule[-.1cm]{0cm}{.4cm} \parbox{.85\textwidth}{\begin{center}
						\texttt{#1}\end{center}} \rule {.25cm}{0cm}}\end{center}}
	{}
}
\theoremstyle{plain} \newtheorem{Theorem}{Theorem}
\theoremstyle{plain} 
\theoremstyle{plain} \newtheorem{Lemma}{Lemma}
\theoremstyle{plain} \newtheorem{Proposition}{Proposition}
\theoremstyle{definition} \newtheorem{Definition}{Definition}
\theoremstyle{definition}\newtheorem{Remark}{Remark}
\theoremstyle{definition}
\theoremstyle{definition} 

\newcommand{\eps}{\varepsilon}

\newcommand{\al}{\alpha}

\newcommand{\vsp}[1]{\vspace{0.5cm}\par}

\newcommand{\baf}{\bar{F}}
\newcommand{\bai}{\bar{\eta}}

\numberwithin{equation}{section}

\newcommand{\ra}{\rangle}
\newcommand{\la}{\langle}
\newcommand{\lf}{\lambda_{F}}

\newcommand{\lv}{\lambda_{v}}
\newcommand{\lvi}{\lambda_{v_i}}
\newcommand{\li}{\lambda_{\eta}}

\newcommand{\bn}{{\nu}}
\newcommand{\bg}{\gamma}

\allowdisplaybreaks

\begin{document}
	\title[Quasiconvex thermoelasticity]{Weak-strong uniqueness for measure-valued solutions to the equations of quasiconvex adiabatic thermoelasticity}

	\author[M. Galanopoulou]{Myrto Galanopoulou}
	\address[Myrto Galanopoulou]{\newline
		Department of Mathematics, Heriot-Watt University\\ 
		Edinburgh, Scotland, UK}
	\email[]{\href{m.galanopoulou@hw.ac.uk}{m.galanopoulou@hw.ac.uk}}

	\author[A. Vikelis]{Andreas Vikelis}
	\address[Andreas Vikelis]{\newline
		Department of Mathematics, University of Sussex\\ Pevensey 2 Building,
		Falmer,
		Brighton, BN1 9QH, UK}
	\email[]{\href{a.vikelis@sussex.ac.uk}{a.vikelis@sussex.ac.uk}}
	
	\author[K. Koumatos]{Konstantinos Koumatos}
	\address[Konstantinos Koumatos]{\newline
		Department of Mathematics, University of Sussex\\ Pevensey 2 Building,
		Falmer,
		Brighton, BN1 9QH, UK}
	\email[]{\href{k.koumatos@sussex.ac.uk}{k.koumatos@sussex.ac.uk}}
	
	\begin{abstract}
		This article studies the equations of adiabatic thermoelasticity endowed with an internal energy satisfying an appropriate quasiconvexity assumption which is associated to the symmetrisability condition for the system. A G\aa rding-type inequality for these quasiconvex functions is proved and used to establish a weak-strong uniqueness result for a class of dissipative measure-valued solutions. 
		\vspace{0.1cm}
		
		\vspace{0.1cm}
		
	\end{abstract}
	
	\maketitle

	\section{Introduction}
	In this article, we consider the system of adiabatic thermoelasticity in Lagrangian coordinates given by 
	\begin{align}
		\begin{split}
			\label{e:adiabatic_thermoelasticity}
			\partial_t F_{i\alpha}-\partial_{\alpha}v_i&=0 \\ 
			\partial_t v_i-\partial_{\alpha}\Sigma_{i\alpha}&=0\\
			\partial_t\left(\frac{1}{2}|v|^2+e\right)-\partial_{\alpha}(\Sigma_{i\alpha}v_i)&=r ,
		\end{split}
	\end{align}
	that describes the evolution of a thermomechanical process $\big ( y(t,x) , \eta(t,x) \big) \in \mathbb{R}^3\times\mathbb{R}^+$  
	where the time variable $t\in\mathbb{R}^+$ and the spatial variable $x\in\mathbb{R}^3$. This is a first-order system and a solution to
	(\ref{e:adiabatic_thermoelasticity}),  consists of the deformation gradient $F = \nabla y  \in \mathbb{M}^{3\times3}$, the velocity $v = 
	\partial_t y \in\mathbb{R}^3$ and the specific entropy $\eta$. The first equation is a compatibility relation between the partial derivatives 
	of the motion, the second describes the balance of linear momentum, while the third equation stands for the balance of energy.	One 
	must include to \eqref{e:adiabatic_thermoelasticity} the constraint 
	\begin{equation}
		\label{constraint}
		\partial_{\alpha}F_{i\beta}=\partial_{\beta}F_{i\alpha}, \qquad i,\alpha,\beta=1,2,3 \, ,
	\end{equation}
	which guarantees that $F$ is  a  deformation gradient associated to the motion $y(t,x)$. We note that relation \eqref{constraint} is an involution, namely, it is  propagated from the initial data to the solution via \eqref{e:adiabatic_thermoelasticity}$_1$.
	
	The remaining variables in (\ref{e:adiabatic_thermoelasticity}) are the Piola-Kirchhoff stress  $\Sigma_{i\alpha }$, the internal energy $e$, 
	and the radiative heat supply $r$.  Here, the referential heat flux $Q_\alpha = 0$ as this theory describes adiabatic processes,  and it 
	does not appear in the equations \eqref{e:adiabatic_thermoelasticity}. For simplicity we have normalised the reference density 
	$\rho_0 =1$. The balance of entropy holds identically as an equality for strong solutions, that is
	\begin{equation}
		\label{e:entropy.prod}
		\partial_t \eta=\frac{r}{\theta(F,\eta)}
	\end{equation}
	and it can be derived from system \eqref{e:adiabatic_thermoelasticity}.
	By contrast, for weak solutions, \eqref{e:entropy.prod} is replaced by the Clausius-Duhem inequality \cite{CN1963,TN1992,MR3468916}, according to the 
	second law of thermodynamics,  and it serves as a criterion of admissibility for thermodynamic processes that satisfy the balance laws of mass, momentum and energy. The system is closed through constitutive relations which, 
	for smooth processes, are consistent with the Clausius-Duhem inequality and describe the material response. For thermoelastic 
	materials under adiabatic conditions,  the constitutive theory is determined from the thermodynamic potential of the internal energy 
	depending solely on the deformation gradient $F$ and the entropy $\eta,$ via the relations
	\begin{align}
		\label{e:constitutive_functions}
		e=e(F,\eta),\quad
		\Sigma=\frac{\partial e}{\partial F},\quad
		\theta=\frac{\partial e}{\partial\eta}
	\end{align}
	for the stress $\Sigma$ and the temperature $\theta$.  We refer the reader to \cite{CN1963,TN1992} for a detailed derivation of adiabatic thermoelasticity and its relation to other constitutive theories.
	
	\textcolor{black}{
		System \eqref{e:adiabatic_thermoelasticity} belongs to a general class of
		hyperbolic problems 
		that are symmetrisable in the sense of Friedrichs and Lax~\cite{MR285799},
		under appropriate hypotheses. It turns out that symmetrisability is guaranteed by the posivitivity of the matrix
		\begin{equation*}
			\frac{1}{e_{\eta}}
			\begin{pmatrix}
				e_{FF} & 0 & e_{F \eta}
				\\
				0 &  1 & 0
				\\
				e_{F \eta} & 0 & e_{\eta \eta}
			\end{pmatrix}
		\end{equation*}
		which in turn amounts to $e(F, \eta)$ being strongly convex and $\theta(F,\eta)=\frac{\partial  e (F,\eta)}{\partial  \eta} > 0$. We refer
		the reader to Appendix B for a discussion on the connection of thermoelasticity to the general theory of conservation laws for symmetrisable systems.}
	
	Convexity of $e(F,\eta)$ suffices to apply the standard 
	theory of conservation laws to \eqref{e:adiabatic_thermoelasticity}, however, the condition of convexity is too restrictive to encompass a large class of materials.
	A broader notion of convexity is polyconvexity, that is $e(F,\eta) = g(F, {\rm cof}\,F, \det F, \eta)$ for some convex function $g$. For polyconvex energies stability and weak-strong uniqueness results for system \eqref{e:adiabatic_thermoelasticity} have been obtained in \cite{CGT2017,CGT2018,CGT2020}. Note that due to the presence of involutions \eqref{constraint}, the positivity of the matrix appearing in \eqref{hyperbolic(F,eta)} is indeed only required on the cone 
	\[
	\left\{(a\otimes n, v, \eta) : a,\,n,\,v\in\mathbb{R}^3,\,\eta\in\mathbb{R}\right\}
	\]
	amounting to a notion of rank-one convexity for $e(F,\eta)$. \textcolor{black}{ Nevertheless, as it was proved by Dafermos in \cite{Daf}, weak-strong uniqueness for hyperbolic systems of conservation laws with entropies which are convex on the wave cone, can be established under an extra assumption of small local oscilations on the weak solutions. Motivated by recent work in the isothermal problem of elasticity \cite{KS2017}, and extensions to more general PDE constrained conservation laws \cite{KV2020}, we prove a weak-strong uniqueness result for a suitable class of dissipative measure-valued solutions to \eqref{e:adiabatic_thermoelasticity} without the assumption of small oscillations, whenever the internal energy $e$ satisfies a natural quasiconvexity assumption associated to the above cone, i.e.
		\begin{align}\label{eq:qscintro}
			\int_\Omega e(\lambda_1+\nabla\phi,\lambda_2+\psi)-e(\lambda_1,\lambda_2)-e_\eta(\lambda_1,\lambda_2)\psi\gtrsim\int_\Omega|V_p(\nabla\phi)|^2+|V_q(\psi)|^2,
		\end{align}
		where $\phi\in W^{1,p}_0(\Omega)$ and $\psi\in L^q(\Omega)$; see \eqref{eq:aux} and Definition \ref{defqcc} for more details. We note that the quasiconvexity condition \eqref{eq:qscintro} is in a sense equivalent to the classical notion of $({\rm curl},0)$-quasiconvexity, see Remark \ref{remark:Aqc}, and so it arises naturally due to the involutions \eqref{constraint}. As a consequence, our definition is more general than polyconvexity. }
	
	\textcolor{black}{ In the spirit of \cite{KS2017} and \cite{KV2020}, and in order to establish weak-strong uniqueness, we consider the relative quantity
		\begin{align*}
			e(z_1,z_2|\bar{z}_1,\bar{z}_2)=e(z_1,z_2)-e(\bar{z}_1,\bar{z}_2)-e_F(\bar{z}_1,\bar{z}_2)\cdot(z_1-\bar{z}_1)-e_\eta(\bar{z}_1,\bar{z}_2)\cdot(z_2-\bar{z}_2)
		\end{align*}
		and for quasiconvex $e$ we prove a G\"{a}rding type inequality, see Theorem \ref{eq:mmgarding},
		\begin{align}\label{intro:mmgarding}
			\int\Big(|V_p(\nabla\phi)|^2+|V_q(\psi)|^2\Big) 
			\lesssim \int e(\baf+\nabla\phi,\bai+\psi|\baf,\bai)+ \int|V_p(\phi)|^2.
		\end{align}
		In comparison with \cite{KV2020}, in our case, the fact that part of the solution remains unconstrained allows us to disregard the corresponding lower order terms in \eqref{intro:mmgarding}, and only retain lower order terms associated with the constrained part. The work presented in this paper borrows newly developed techniques from the calculus of variations \cite{Acerbi1987, Zhang1992, JC2017, JCK2020} and constitutes the first result for the hyperbolic system of thermoelasticity under quasiconvexity conditions.}
	
	The article is organised as follows: in Section 2, we introduce the appropriate notions of quasiconvexity and measure-valued solutions for system \eqref{e:adiabatic_thermoelasticity}. In Section 3, we present the relative entropy inequality which is applied in Section 4 to prove measure-valued vs strong uniqueness for the system of quasiconvex adiabatic thermoelasticity. The proof requires a G\aa rding-type inequality for these quasiconvex functions which is the content of Section 6, while Section 5 establishes some necessary localisation (in time) results for these measure-valued solutions. \textcolor{black}{Finally, Appendix A contains some auxiliary results, while Appendix B discusses how system \eqref{e:adiabatic_thermoelasticity} fits into the general theory of symmetrisable systems of conservation laws.}

	\section{Preliminaries}
	\subsection{Quasiconvexity:} Let $Q$ be (any translation of) the unit cube $(0,1)^d$ with opposite boundaries glued or simply the $d$-dimensional torus $\mathbb{T}^d$. In the sequel, since we often consider functions defined on cubes which are seen as subsets of the unit cube $Q$, we prefer to define our integrals and functional spaces over $Q$ with opposite sides identified. That is, we write $L^p(Q)$ instead of $L^p(\mathbb{T}^d)$ but also $C^k(Q)$ instead of $C^k(\mathbb{T}^d)$, i.e.
	\begin{equation*}\label{eq:QvT}
		C^k({Q}):=\left\{\psi\in C^k(\mathbb{R}^d) : \partial^\alpha \psi \text{ $\Delta$-periodic for all $d$-multi-indices $|\alpha|\leq k$} \right\}
	\end{equation*}
	where $\Delta$ denotes the unit cell of the lattice $\mathbb{Z}^d$.
	
	In the spirit of \cite{KS2017}, \cite{KV2020}, we replace convexity of the free energy by a notion of quasiconvexity which we define below, Definition \ref{defqcc}. The system \eqref{e:adiabatic_thermoelasticity} together with \eqref{hyperbolic(F,eta)} indicate that our notion of convexity should at least imply convexity on the directions of the wave cone 
	\begin{equation*}
		\Lambda:=\left\{(a\otimes n, \eta) : a,\,n,\in\mathbb{R}^3,\,\eta\in\mathbb{R}\right\}.
	\end{equation*}
	
	To this end, we say that a continuous function $e$ is quasiconvex at $(\lambda_1,\lambda_2)\in\mathbb{R}^{d\times d}\times\mathbb{R}$ if the inequality
	\begin{equation*}
		\int_Q e(\lambda_1+\nabla\phi(x),\lambda_2+\psi(x))-e(\lambda_1,\lambda_2)-e_\eta(\lambda_1,\lambda_2)\psi(x) dx\geq 0,
	\end{equation*}
	holds for all $\phi\in W^{1,\infty}(Q)$ and $\psi\in L^\infty(Q)$. The above definition can equivalently be expressed over arbitrary domains. In particular, for $\Omega\subseteq\mathbb{R}^d$ a non-empty open subset with $|\partial\Omega| = 0$, $e$ is quasiconvex at $(\lambda_1,\lambda_2)\in\mathbb{R}^{d\times d}\times\mathbb{R}$ if the inequality
	\begin{equation*}
		\int_\Omega e(\lambda_1+\nabla\phi(x),\lambda_2+\psi(x))-e(\lambda_1,\lambda_2)-e_\eta(\lambda_1,\lambda_2)\psi(x)dx\geq 0,
	\end{equation*}
	holds for all $\phi\in W^{1,\infty}_0(\Omega)$ and $\psi\in L^\infty(\Omega)$.
	
	\vspace{0.2cm}
	\noindent Henceforth, we assume that $e$ has $(p,q)$-growth, i.e. $|e(z_1,z_2)|\leq c(1+|z_1|^p+|z_2|^q)$, and then by density we can also express the above definition with test functions in $W^{1,p}(Q)$ and $L^q(Q)$ or in $W^{1,p}_0(\Omega)$ and $L^q(\Omega)$ respectively. 
	
	\vspace{0.2cm}
	The results presented in this paper require a strengthened version of quasiconvexity which we now introduce. Let $p, q\geq 2$ and  define the auxiliary function $V_i:\mathbb{R}^k\to\mathbb{R}$ by
	\begin{equation}\label{eq:aux}
		V_i(z):=(|z|^i+|z|^2)^{1/2},
	\end{equation}
	where $k=1,\,k=d$ or $k=d\times d$ and $i\in\mathbb{N}$. 
	\begin{Definition}\label{defqcc}
		Let $\Omega\subseteq\mathbb{R}^d$ be a non-empty open subset with $|\partial\Omega| = 0$. A continuous function $e$ is strongly quasiconvex at $(\lambda_1,\lambda_2)\in\mathbb{R}^{d\times d}\times\mathbb{R}$ if the inequality
		\begin{equation*}
			\int_\Omega e(\lambda_1+\nabla\phi,\lambda_2+\psi)-e(\lambda_1,\lambda_2)-e_\eta(\lambda_1,\lambda_2)\psi\geq c_0\int_\Omega|V_p(\nabla\phi)|^2+|V_q(\psi)|^2,
		\end{equation*}
		holds for all $\phi\in W^{1,p}_0(\Omega)$ and $\psi\in L^q(\Omega)$. Equivalently, $e$ is strongly quasiconvex at $(\lambda_1,\lambda_2)\in\mathbb{R}^{d\times d}\times\mathbb{R}$ if 
		\begin{equation*}
			\int_Q e(\lambda_1+\nabla\phi,\lambda_2+\psi)-e(\lambda_1,\lambda_2)-e_\eta(\lambda_1,\lambda_2)\psi\geq c_0\int_Q|V_p(\nabla\phi)|^2+|V_q(\psi)|^2,
		\end{equation*}
		for all $\phi\in W^{1,p}(Q)$ and $\psi\in L^q(Q)$.  In addition, we say that $e$ is (strongly) quasiconvex if it is (strongly) quasiconvex at $(\lambda_1,\lambda_2)$ for all $(\lambda_1,\lambda_2)\in\mathbb{R}^{d\times d}\times\mathbb{R}$.
	\end{Definition}
	\begin{Remark}\label{remark:Aqc}
		At first sight, our definition of quasiconvexity, Definition \ref{defqcc}, differs from the classical definition of $({\rm curl},0)$-quasiconvexity associated with the cone $\Lambda$, see \cite{FM1999} for more details. We remark that, with respect to the latter definition, $e$ is strongly $({\rm curl},0)$-quasiconvex if 
		\begin{equation}\label{eq:Aqc}
			\int_\Omega e(\lambda_1+\nabla\phi,\lambda_2+\psi)-e(\lambda_1,\lambda_2)dx\geq c_0\int_\Omega|V_p(\nabla\phi)|^2+|V_q(\psi)|^2dx,
		\end{equation}
		for all $\phi\in W^{1,p}_0(\Omega)$ and $\psi\in L^q(\Omega)$ with $\psi|_\Omega:=\int_\Omega\psi=0$. In fact, the two definitions are in a sense equivalent. In particular, if $e$ is  quasiconvex for all $(\lambda_1,\lambda_2)\in\mathbb{R}^{d\times d}\times\mathbb{R}$ then it is also strongly $({\rm curl},0)$-quasiconvex up to a possibly different constant $c_0$ (the reverse is obvious). Indeed, write
		\begin{align*}
			&e(\lambda_1+\nabla\phi,\lambda_2+\psi)-e(\lambda_1,\lambda_2)-e_\eta(\lambda_1,\lambda_2)\psi \\
			&=e(\lambda_1+\nabla\phi,\lambda_2+\psi|_\Omega+\psi-\psi|_\Omega)
			-e(\lambda_1,\lambda_2+\psi|_\Omega)\\
			&+e(\lambda_1,\lambda_2+\psi|_\Omega)-e(\lambda_1,\lambda_2)-e_\eta(\lambda_1,\lambda_2)\psi|_\Omega-e_\eta(\lambda_1,\lambda_2)\psi+e_\eta(\lambda_1,\lambda_2)\psi|_\Omega,
		\end{align*}
		and then use the inequality \eqref{eq:Aqc} together with the fact that, for all $(\lambda_1,\lambda_2)\in\mathbb{R}^{d\times d}\times\mathbb{R}$ and $\xi\in\mathbb{R}$, it holds that
		\[
		e(\lambda_1,\lambda_2+\xi)-e(\lambda_1,\lambda_2)-e_\eta(\lambda_1,\lambda_2)\xi\geq c_0|V_q(\xi)|^2.
		\]
		The latter arises due to the convexity of $e$ in $\lambda_2$, according to Definition \ref{defqcc}.
	\end{Remark}
	
	\vspace{0.2cm}
	
	\subsection{Dissipative measure-valued solutions:}
	For $p\geq 0$, let $C_p(\mathbb{R}^d)$ denote the space of continuous functions such that
	\[C_p(\mathbb{R}^d):=\left\{g\in C(\mathbb{R}^d):\lim_{|\lambda|\to\infty}\frac{g(\lambda)}{|\lambda|^p}=0\right\}\]
	while the space $C_0(\mathbb{R}^d)$ is defined as
	\[C_0(\mathbb{R}^d):=\left\{g\in C(\mathbb{R}^d):\lim_{|\lambda|\to\infty} g(\lambda)=0\right\}.\]
	Identifying the space of signed Radon measures $\mathcal{M}(\mathbb{R}^d)$ equipped with the total variation norm as 
	isometrically isomoprhic to the dual of $C_0(\mathbb{R}^d)$, a (parametrised) Young measure $\nu=(\nu_x)_{x\in Q}$ is an element of the space 
	$L^{\infty}_{w*}(Q,\mathcal{M}(\mathbb{R}^d))$ taking values in the space of probability measures. The space 
	$L^{\infty}_{w*}(Q,\mathcal{M}(\mathbb{R}^d))$ consists of all weak-$*$ measurable, essentially bounded maps 
	$\nu:Q\ni x\mapsto\nu_x\in\mathcal{M}(\mathbb{R}^d)$, i.e. all maps such that 
	\[x\mapsto \la\nu,g\ra:=\int g(\lambda)\:\nu(d\lambda)\] is measurable for all $g\in C_0(\mathbb{R}^d)$ and
	\[\sup_{x\in Q}\|\nu_x\|_{\mathcal{M}(\mathbb{R}^d)}<\infty.\]
	Since $C_0(\mathbb{R}^d)$ is separable, we have
	\[L^{\infty}_{w*}(Q,\mathcal{M}(\mathbb{R}^d))=L^1(Q,C_0(\mathbb{R}^d))^{\ast}\] and this defines the weak-$\ast$ limits of
	Young measures. The Fundamental Theorem of Young measures in $L^p$ states that given a bounded sequence $(U_n)$ in
	$L^p(Q)$ $(1\leq p<\infty),$ there exists a subsequence and a parametrized family of Young measures $\nu=(\nu_x)_{x\in Q}$
	such that 
	\begin{align}
		\label{g_weaklim_definition}
		g(U_n) \rightharpoonup\la\nu_x,g\ra \quad\text{in}\; L^1(Q), \; \forall \; g\in C_p(\mathbb{R}^d),
	\end{align}
	and we say that the sequence $(U_n)$ generates the Young measure $\nu.$ We call $\nu$ a \textit{$p-$Young measure} since it is 
	generated by a bounded sequence in $L^p.$ In fact, the sequence $\left(g(U_n)\right)$ converges as in \eqref{g_weaklim_definition} whenever it is 
	equiintegrable and the barycentre $\la\nu_x,id\ra$ of the generated Young measure gives the weak limit of the sequence $U_n$, i.e.
	\begin{align*}
		U_n \rightharpoonup\la\nu_x,id\ra \quad\text{in}\; L^p(Q).
	\end{align*}
	If $U_n=\nabla u_n$ for $u_n\in W^{1,p}(Q),$ then we call $\nu$ a \textit{gradient $p-$Young measure}.
	Below, we wish to consider generating sequences $(U_n)\subset L^{\infty}(0,T;L^p(Q))$, defined both in time and space, for some $T>0$. Then, the associated Young 
	measure $\nu_{t,x}$, with $(t,x)\in Q_T := (0,T)\times Q$, satisfies 
	\[\sup_{0\leq t\leq T}\int \la\nu_{t,x},|\cdot|^p\ra<\infty.\]
	In our context, we naturally consider measure-valued solutions as limits of approximations that satisfy the uniform bound
	\begin{align}
		\label{energy_unif_bound}
		\sup_{0\leq t\leq T} \int_{\mathbb{T}^d} e(F^\varepsilon,\eta^\varepsilon)+\frac{1}{2}|v^\varepsilon|^2\:dx\leq C,
	\end{align}
	coming by integrating in $Q_T$ the energy conservation equation (\ref{e:adiabatic_thermoelasticity})$_3,$ given that the radiative heat supply $r$ is
	bounded in $L^1(Q_T)$ and that the initial data have bounded energy.
	Assuming the growth condition on the energy: 
	\begin{align}
		\label{H2}
		c(|F|^p+|\eta|^q-1)\leq e(F,\eta)\leq c(|F|^p+|\eta|^q+1),
	\end{align}
	then growth \eqref{H2} together with \eqref{energy_unif_bound}, implies the following uniform in $\varepsilon$ bounds:
	\begin{equation}
		\label{Fvtheta.reg}
		F^{\varepsilon}\in L^{\infty}(0,T;L^p(Q)) ,\quad v^{\varepsilon}\in L^{\infty}(0,T;L^2(Q)), \quad \eta^{\varepsilon}\in L^{\infty}(0,T;L^q(Q)).
	\end{equation}
	Then, the sequence $\{(F^\varepsilon,v^\varepsilon,\eta^\varepsilon)\}$ generates a family of probability measures
	$$\nu_{t,x}\in\mathcal{P}(\mathbb{M}^{3\times 3}\times\mathbb{R}^3\times\mathbb{R})$$
	given by the mapping $(\nu_{t,x}):Q_{T}\ni(t,x)\mapsto\nu_{t,x}.$ The Young measure $(\nu_{t,x})$ is an element of the space
	$L^{\infty}_{w*}(Q_T,\mathcal{M}(\mathbb{R}^{13}))$ representing weak limits of the form
	\begin{align}
		\label{psi_weaklim_definition}
		\text{wk-$\ast$-}\lim_{\varepsilon\to 0}\psi(F^\varepsilon,v^\varepsilon,\eta^\varepsilon)
		=\la\nu_{t,x},\psi(\lf,\lv,\li)\ra,
	\end{align}
	for all continuous functions $\psi=\psi(\lf,\lv,\li)$ such that
	\begin{align}
		\label{psi_weaklim_growth}
		\lim_{|\lf|^p+|\lv|^2+|\li|^q\to\infty}
		\frac{|\psi(\lf,\lv,\li)|}{|\lf|^p+|\lv|^2+|\li|^q}=0,
	\end{align}
	where in \eqref{psi_weaklim_definition} the notation $\la\nu_{t,x},\cdot\ra$ stands for the average
	$$\la\nu_{t,x},\psi(\lf,\lv,\li)\ra=\int\psi(\lf,\lv,\li)\:\nu_{t,x}(d\lf,d\lv,d\li)$$ 
	and $\lf\in\mathbb{M}^{3\times 3},$ $\lv\in\mathbb{R}^3,$ $\li\in\mathbb{R}.$ The marginal of $\nu_{t,x}$ generated by 
	$(F^{\varepsilon})_{\varepsilon}=(\nabla y^{\varepsilon})_{\varepsilon}$ is a gradient $p-$Young measure, while the marginals generated 
	by $(v^\varepsilon)$ and $(\eta^{\varepsilon})$ are a $2$- and a $q$-Young measure respectively. In particular,
	\begin{align*}
		F^{\varepsilon}\overset{\ast}{\rightharpoonup}\la\nu_{t,x},\lambda_F\ra=: F   \quad\text{weak-$\ast$ in\:} L^{\infty}(0,T;L^p(Q)) \;,\\
		v^{\varepsilon}\overset{\ast}{\rightharpoonup}\la\nu_{t,x},\lambda_v\ra =: v  \quad\text{weak-$\ast$ in\:} L^{\infty}(0,T;L^2(Q)) \;,\\
		\eta^{\varepsilon}\overset{\ast}{\rightharpoonup}\la\nu_{t,x},\lambda_{\eta}\ra =: \eta \quad\text{weak-$\ast$ in\:} L^{\infty}(0,T;L^q(Q)) \, .
	\end{align*}
	We note that the space $C_p(\mathbb{R}^d)$ is separable 
	equipped with the norm $\|g(\cdot)/(1+|\cdot|^p)\|_{L^{\infty}},$ and so is the space $C_{p,q}(\mathbb{R}^d\times\mathbb{R})$ defined as
	\[C_{p,q}(\mathbb{R}^d\times\mathbb{R}):=\left\{g\in C(\mathbb{R}^d\times\mathbb{R}):
	\lim_{|\lambda_1|^p+|\lambda_2|^q\to\infty}\frac{g(\lambda_1,\lambda_2)}{|\lambda_1|^p+|\lambda_2|^q}=0\right\},\]
	equipped with the norm $\|g(\cdot)/(1+|\cdot|^p+|\cdot|^q)\|_{L^{\infty}}.$ As a result, the internal energy function $e(\lambda_{F},\li)$ 
	belongs to the separable space $C_{p,q}(\mathbb{R}^9\times\mathbb{R})$ ($\mathbb{M}^{3\times3} \simeq \mathbb{R}^9$) under the 
	aformentioned norm. 
	%
	
	To take into account the formation of concentration effects, we introduce the concentration measure $\bg$, that depends on the total energy. This is a well-defined nonnegative Radon measure for a subsequence of 
	$$ \frac{1}{2}|v^\varepsilon|^2+e(F^\varepsilon,\eta^\varepsilon).$$  Since we know that the functions $(F^{\varepsilon},v^{\varepsilon},\eta^{\varepsilon})$ are all bounded in some $L^p$ space -because of \eqref{Fvtheta.reg}- we may 
	apply the generalized Young measure Theorem \cite{dm87,ab97}, in order to pass to the limit, since the only 
	assumption required is $L^1$ boundedness. Indeed, letting $\Omega$ be an open subset of $\mathbb{R}^d$, the theorem asserts that given a 
	sequence of functions $(u_n),$ $u_n:\Omega\to\mathbb{R}^m,$ bounded in $L^p(\Omega),$ ($p\geq 1$) there exists a 
	subsequence (which we will not relabel), a parametrized family of probability measures 
	$\nu\in L^{\infty}_{w*}(\Omega;\mathcal{P}(\mathbb{R}^m)),$ a nonnegative measure 
	${\mu}\in\mathcal{M}^+(\Omega)$ and a parametrized probability measure on a 
	sphere $\nu^{\infty}\in L^{\infty}_{w*}((\Omega,{\mu});\mathcal{P}(S^{m-1}))$ such that
	\begin{align}
		\label{ab97_weaklimits}
		\psi(x,u_n)\:dx\rightharpoonup\int_{\mathbb{R}^m}\psi(x,\lambda)\:d{\nu}dx
		+ \int_{S^{m-1}}\psi^{\infty}(x,z)\:{\nu}^{\infty}(dz) d{\mu}
		\quad\text{weakly-$\ast$,}
	\end{align}
	for all $\psi$ continuous with well-defined recession function 
	$$\psi^{\infty}(x,z):=\lim_{\substack{s\to\infty\\z'\to z}}\frac{\psi(x,sz')}{s^p}.$$ 
	
	The sequences $(F^{\varepsilon},v^{\varepsilon},\eta^{\varepsilon})$ are bounded in different spaces and have different growth, 
	and as a result, we need to apply a refinement of the aforementioned theorem as, for instance, in \cite{GSW2015}: 
	consider a sequence of maps $u_n=(u^1_n,u^2_n)$ where $(u^1_n)$ is bounded in some 
	$L^p(\Omega;\mathbb{R}^b)$ and $(u^2_n)$ is bounded in $L^q(\Omega;\mathbb{R}^l)$ and define the 
	non-homogeneous unit sphere 
	$$S^{b+l-1}_{pq}:=\{(\beta_1,\beta_2)\in\mathbb{R}^{b+l}:|\beta_1|^{p}+|\beta_2|^{q}=1\}\,,$$
	for exponents $p, q > 1.$
	Then one can pass to the limit as in \eqref{ab97_weaklimits} where 
	\[\begin{aligned}
		\psi^{\infty}(x,z) 
		&:=\lim_{\substack{x'\to x\\s\to\infty\\(\beta_1',\beta_2')\to (\beta_1,\beta_2)}}
		\frac{\psi(x',s^q\beta_1',s^p\beta_2')}{s^{pq}}
		= \lim_{\substack{x'\to x\\ \tau \to\infty\\(\beta_1',\beta_2')\to (\beta_1,\beta_2)}}
		\frac{\psi(x', \tau^{\frac{1}{p}} \beta_1', \tau^{\frac{1}{q}} \beta_2')}{\tau}.
	\end{aligned}\]
	We define the generalized sphere
	$$
	S^{12} = \{ (F,v,\eta) \in \mathbb{R}^{13} : |F|^p + |v|^2 + |\eta|^q = 1 \} \, .
	$$
	The form of the recession function for the energy follows from \cite[Thm 2.5]{ab97} and reads
	\begin{align*}
		\left(\frac{1}{2}|v|^2+e(F,\eta)\right)^{\infty}=\lim_{\tau \to\infty}
		\left(\frac{1}{2}|v|^2+\frac{e( \tau^{\frac{1}{p}} F ,  \tau^{\frac{1}{q}} \eta)}{\tau}\right),
	\end{align*}
	and we require it to be continuous on $S^{12}$.  Then, along a subsequence in $\varepsilon$,
	\begin{align*}
		\frac{1}{2}|v^{\varepsilon}|^2&+e(F^\varepsilon,\eta^\varepsilon) \overset{\ast}{\rightharpoonup}\\
		&\overset{\ast}{\rightharpoonup}
		\left\la\nu_{t,x},\frac{1}{2}|\lv|^2+e(\lf,\li)\right\ra \, dx
		+ \left\la\nu^{\infty},\left(\frac{1}{2}|\lv|^2+e(\lf,\li)\right)^{\infty}\right\ra{\mu}
	\end{align*}
	weak-$\ast$ in the sense of measures, where $\nu\in\mathcal{P}(Q_{T};\mathbb{R}^{13}),$ 
	$\nu^{\infty}\in\mathcal{P}((Q_{T},{\mu});S^{12})$ and ${\mu}\in\mathcal{M}^+(Q_T).$ 
	Then \eqref{H2} implies $$\left(\frac{1}{2}|\lv|^2+e(\lf,\li)\right)^{\infty}>0$$
	so that the concentration measure $\gamma \in \mathcal{M}^+(Q_T)$ is nonnegative, i.e.
	\begin{align}
		\label{concmeas}
		\bg:=\left\la\nu^{\infty},\left(\frac{1}{2}|\lv|^2+e(\lf,\li)\right)^{\infty}\right\ra{\mu} \ge 0 \, .
	\end{align}
	
	The following definition of a measure-valued solution for system \eqref{e:adiabatic_thermoelasticity} thus arises:
	
	\begin{Definition} \label{MVS_adiabatic_theromelasticity}
		A dissipative measure-valued solution to adiabatic thermoelasticity \eqref{e:adiabatic_thermoelasticity}, \eqref{e:entropy.prod}
		consists of a thermomechanical process $(y(t,x),\eta(t,x)):[0, T]\times Q\to\mathbb{R}^3\times\mathbb{R}$ for any $T > 0$,
		\begin{align}
			\label{e:regulatity.y} 
			y \in W^{1,\infty}(0,T;L^2(Q)) \cap L^\infty (0,T;W^{1,p} (Q)) \, , \quad \eta\in L^{\infty}(0,T;L^{q} (Q)   ) \,  ,
		\end{align}
		a parametrized family of probability measures $\nu=(\nu_{t,x})_{(t,x)\in Q_T}$ and a nonnegative Radon measure 
		$\bg\in\mathcal{M}^+(Q_{T}).$ \textcolor{black}{The measure $\nu$ is generated by a sequence $(v^\varepsilon,\nabla y^\varepsilon,\eta^\varepsilon)$ 
			such that 
			\begin{align}\label{H-1gen}
				\begin{split}
					&(y^\varepsilon)\hspace{0.2cm}\text{is bounded in}\hspace{0.2cm}L^\infty(0,T;W^{1,p}(Q))\\
					&(\partial_t\nabla y^\varepsilon) \hspace{0.2cm}\text{is bounded in}\hspace{0.2cm}L^\infty(0,T;H^{-1}(Q))\\
					&(\eta^\varepsilon)\hspace{0.2cm}\text{is bounded in}\hspace{0.2cm}L^\infty(0,T;L^{q}(Q)).
				\end{split}
		\end{align}}
		If $(v,F,\eta)$ denote the averages
		\begin{align*}
			F=\left\la\nu_{t,x},\lambda_F \right\ra, \quad v=\left\la\nu_{t,x},\lv\right\ra, \quad \eta=\left\la\nu_{t,x},\li\right\ra \, , 
		\end{align*} 
		then $\nu_{t,x}$ and $\bg$ satisfy
		\begin{align}
			\label{e:regulatity.xi-v-eta}
			F=\nabla y\in L^{\infty}(L^p), \quad  v=\partial_t y\in L^{\infty}(L^2) \, , 
		\end{align}
		and the equations
		\begin{align}
			\begin{split}
				\label{e:mv.sol.xi-v-eta}
				\partial_t F&=\partial_{\alpha}v_i \\ 
				\partial_t \left\la\nu_{t,x},\lvi\right\ra
				&=\partial_{\alpha}\left\la\nu_{t,x},\frac{\partial e}{\partial F_{i\alpha}}(\lambda_F,\lambda_{\eta})\right\ra\\
				\partial_t \left\la\nu_{t,x},\li\right\ra &\geq \left\la\nu_{t,x},\frac{r}{\theta(\lambda_F,\li)}\right\ra
			\end{split}
		\end{align}
		in the sense of  distributions. Moreover, they satisfy the integrated form of the  averaged energy inequality,
		\begin{small}
			\begin{align}
				\begin{split}
					\label{e:mv.sol.Fvt.energy}
					\int& \varphi(0)\left\la\nu_{0,x},\frac{1}{2}|\lv|^2+e(\lambda_F,\li)\right\ra\:dx \\ 
					&\int_0^T
					\int  \varphi^{\prime}(t) \left(\left\la\nu_{t,x},\frac{1}{2}|\lv|^2+e(\lambda_F,\li) \right\ra(t,x)  \, dx\:dt +\bg(dx\:dt) \right)  \\
					&=-\int_0^T\int \left\la\nu_{t,x},r\right\ra\varphi(t)\:dx\:dt,
				\end{split}
			\end{align}
		\end{small}
		holding for all $\varphi \in C^{\infty}_c([0,T))$, $\varphi\ge 0$.
	\end{Definition}
	
	\begin{Remark}\rm
		On the definition of the dissipative measure-valued solution:
		\begin{enumerate}
			\item \textcolor{black}{We remark that in addition to the uniform estimate \eqref{energy_unif_bound}, natural approximations
				of \eqref{e:adiabatic_thermoelasticity}, \eqref{e:entropy.prod} produce a uniform bound on the time derivatives of $(F^\varepsilon)$ 
				and $(v^\varepsilon)$ in a negative Sobolev space. We take all this into account by assuming \eqref{H-1gen}. }
			\item  The first equation holds in a classical weak form, due to its linearity.
			\item  Henceforth, we  assume the measure $\bg_0=0,$ meaning that we consider initial data with no 
			concentrations at time $t=0.$ 
			\item We choose to work with dissipative measure-valued solutions, namely solutions that satisfy the integrated form of the averaged 
			energy equation \eqref{e:mv.sol.Fvt.energy}. This approach has the technical advantage that one does not need to place any integrability condition on the right hand-side of the energy equation \eqref{e:adiabatic_thermoelasticity}$_3$, namely on the term 
			$\Sigma_{i\alpha} v_i,$ since it appears as a divergence and its contribution integrates to zero.
		\end{enumerate}
	\end{Remark}
	
	\section{The averaged relative entropy inequality} \label{sec_RE}
	Consider a strong solution $(\bar{F},\bar{v},\bar{\eta})^T \in W^{1,\infty}(Q_T)$ to \eqref{e:adiabatic_thermoelasticity}
	that satisfies the entropy identity \eqref{e:entropy.prod} and a dissipative measure-valued solution to 
	\eqref{e:adiabatic_thermoelasticity},~\eqref{e:entropy.prod} according to Definition \ref{MVS_adiabatic_theromelasticity}. 
	We write the difference of the weak form of 
	equations \eqref{e:adiabatic_thermoelasticity} and (\ref{e:mv.sol.xi-v-eta}), to obtain the following	three integral identities
	\begin{equation}
		\begin{split}
			\label{eq1.wk.mv}
			\int(F_{i\alpha}-\bar{F}_{i\alpha})(0,x)\phi_1(0,x)\:dx 
			&+\int_0^T \int(F_{i\alpha}-\bar{F}_{i\alpha})\partial_t\phi_1(t,x)\:dx\:dt\\
			&=\int_0^T \int (v_i-\bar{v}_i)\partial_{\alpha}\phi_1(t,x)\:dx\:dt,
		\end{split}
	\end{equation}
	\begin{equation}
		\begin{split}
			\label{eq2.wk.mv}
			\int(\left\la\nu_{0,x},\lvi\right\ra-&\bar{v}_i(0,x))\phi_2(0,x)\:dx
			+\int_0^T\int (\left\la\nu_{t,x},\lvi\right\ra-\bar{v}_i)\partial_t\phi_2(t,x)\:dx\:dt\\
			&=\int_0^T\int \left(\left\la\nu_{t,x},\Sigma_{i\alpha}(\lf,\li)\right\ra-\Sigma_{i\alpha}(\bar F,\bar\eta)\right)\partial_{\alpha}\phi_2(t,x)\:dx\:dt\;,
		\end{split}
	\end{equation}
	and
	\begin{align}
		\label{eq3.wk.mv}
		&\int\left(\left\la\nu_{0,x},\frac{1}{2}|\lv|^2+e(\lf,\li)\right\ra
		-\left(\frac{1}{2}|\bar{v}|^2+e(\bar{F},\bar{\eta})\right)(0,x)\right)\:\phi_3(0,x)\:dx\nonumber \\
		&\;+\int_0^T \int \Bigg\{\left(\left\la\nu_{t,x},\frac{1}{2}|\lv|^2+e(\lf,\li)\right\ra
		-\frac{1}{2}|\bar{v}|^2-e(\bar{F},\bar{\eta})\right)
		+\bg\Bigg\}\partial_t\phi_3(t,x)\:dx\:dt\nonumber \\
		&=-\int_0^T \int (\left\la\nu_{t,x},r\right\ra-\bar{r})\phi_3(t,x)\:dx\:dt,
	\end{align}
	for any $\phi_i\in C^1_c([0,T)\times Q)$, $i=1, 2$ and $\phi_3\in C^1_c([0,T)).$ Similarly, testing the difference of 
	(\ref{e:entropy.prod}) and (\ref{e:mv.sol.xi-v-eta})$_3$ against $\phi_4\in C^1_c([0,T)\times Q),$  with $\phi_4\ge 0$, we have
	\begin{equation}
		\label{entr1.wka.mv}
		\begin{split}
			-\int(\left\la\nu_{0,x},\li\right\ra-\bar{\eta}(0,x))&\phi_4(0,x)\:dx
			-\int_0^T\int(\left\la\nu_{t,x},\li\right\ra-\bar{\eta})\partial_t\phi_4(t,x)\:dx\:dt\\ &\geq\int_0^T\int\left(\left\la\nu_{t,x},\frac{r}{\theta(\lf,\li)}\right\ra-\frac{\bar{r}}{\theta(\bar{F},\bar{\eta})}\right)\phi_4(t,x)\:dx\:dt.
		\end{split}
	\end{equation}
	We then choose \[(\phi_1,\phi_2,\phi_3)=-\theta(\bar{F},\bar{\eta})\,G(\bar{U})\varphi(t)=( -\Sigma(\bar{F},\bar{\eta}) ,-\bar{v},1)^T\varphi(t),\]
	for some $\varphi\in C_c^1[0,T]$. Thus, by virtue of 
	\eqref{e:constitutive_functions}, equations (\ref{eq1.wk.mv}), (\ref{eq2.wk.mv}) and (\ref{eq3.wk.mv}) become
	\begin{small}
		\begin{align}
			\begin{split}
				\label{eq1phi.wk.mv}
				\int &\left(-\frac{\partial e}{\partial F_{i\alpha}}(\bar{F},\bar{\eta})(F_{i\alpha}-\bar{F}_{i\alpha})\right)(0,x)\varphi(0)\:dx \\
				&+\int_0^T\int\left(-\frac{\partial e}{\partial F_{i\alpha}}(\bar{F},\bar{\theta})(F_{i\alpha}-\bar{F}_{i\alpha})\right)\varphi^{\prime}(t)\:dx\:dt\\
				&=\int_0^T\int \left[\partial_t\Big(\frac{\partial e}{\partial F_{i\alpha}}(\bar{F},\bar{\eta})\Big)(F_{i\alpha}-\bar{F}_{i\alpha})
				-\partial_{\alpha}\Big(\frac{\partial e}{\partial F_{i\alpha}}(\bar{F},\bar{\eta})\Big)(v_i-\bar{v}_i)\right]\varphi(t) \:dx\:dt \, ,
			\end{split}
		\end{align}
	\end{small}
	\begin{equation}
		\begin{split}
			\label{eq2phi.wk.mv}
			\int&-\bar{v}_i(\left\la\nu_{0,x},\lvi\right\ra-\bar{v}_i(0,x))\varphi(0)\:dx
			+\int_0^T\int-\bar{v_i}(\left\la\nu_{t,x},\lvi\right\ra-\bar{v}_i)\varphi^{\prime}(t)\:dx\:dt\\
			&=-\int_0^T\int\left[-\partial_{\alpha}\left(\frac{\partial e}{\partial F_{i\alpha}}(\bar{F},\bar{\eta})\right)
			(\left\la\nu_{t,x},\lvi\right\ra-\bar{v}_i)\right.\\
			&\:\:\hspace{2cm}\left.
			+\partial_{\alpha}\bar{v}_i\left(\left\la\nu_{t,x},\frac{\partial e}{\partial F_{i\alpha}}(\lf,\li)\right\ra-\frac{\partial e}{\partial F_{i\alpha}}(\bar{F},\bar{\eta})\right)\right]\varphi(t)\:dx\:dt\;,
		\end{split}
	\end{equation}
	and
	\begin{align}
		\label{eq3phi.wk.mv}
		\int&\left(\left\la\nu_{0,x},\frac{1}{2}|\lv|^2+e(\lf,\li)\right\ra
		-\left(\frac{1}{2} |\bar{v}|^2-e(\bar{F},\bar{\eta})\right)\!(0,x)\right)\;\varphi(0)\:dx\nonumber \\
		&\;\;+\int_0^T \int \Big\{\left(\left\la\nu_{t,x},\frac{1}{2}|\lv|^2+e(\lf,\li)\right\ra
		-\frac{1}{2}|\bar{v}|^2-e(\bar{F},\bar{\eta})\right)
		+\bg\Big\}\varphi^{\prime}(t)\:dx\:dt\nonumber \\
		&\quad=-\int_0^T \int (\left\la\nu_{t,x},r\right\ra-\bar{r})\varphi(t)\:dx\:dt.
	\end{align}
	For inequality (\ref{entr1.wka.mv}), we choose accordingly $\phi_4:=\theta(\bar{F},\bar{\eta})\varphi(t)\geq 0$, $\varphi \geq 0$ so that
	\begin{align}
		\begin{split}
			\label{entr2.wk.mv}
			-\int&\theta(\bar{F},\bar{\eta})(\left\la\nu_{0,x},\li\right\ra-\bar{\eta}(0,x))\varphi(0)\:dx
			-\int_0^T\int\theta(\bar{F},\bar{\eta})(\left\la\nu_{t,x},\li\right\ra-\bar{\eta})\varphi^{\prime}(t)\:dx\:dt\\
			&\geq\int_0^T\int \Bigg[\partial_t\theta(\bar{F},\bar{\eta})(\left\la\nu_{t,x},\li\right\ra-\bar{\eta})\\
			&\qquad\qquad\qquad
			+\theta(\bar{F},\bar{\eta})\left(\left\la\nu_{t,x},\frac{r}{\theta(\lf,\li)}\right\ra-\frac{\bar{r}}{\theta(\bar{F},\bar{\eta})}\right)\Bigg]\varphi(t)\:dx\:dt.
		\end{split}
	\end{align}
	Adding together (\ref{eq1phi.wk.mv}), (\ref{eq2phi.wk.mv}), (\ref{eq3phi.wk.mv}) and (\ref{entr2.wk.mv}), we obtain the integral inequality
	\begin{align}
		\label{entr.eq.1.mv}
		\int&\varphi(0)\bigg[-\frac{\partial e}{\partial F_{i\alpha}}(\bar{F},\bar{\eta})(F_{i\alpha}-\bar{F}_{i\alpha})(0,x)-\la\nu_{0,x},\bar{v}_i(\lvi-\bar{v}_i)\ra(0,x) \nonumber \\
		&\qquad
		+\left\la\nu_{0,x},\frac{1}{2}|\lv|^2+e(\lf,\li)-\frac{1}{2}|\bar{v}|^2-e(\bar{F},\bar{\eta})\right\ra(0,x) \nonumber \\
		&\qquad\qquad\qquad\qquad\qquad\qquad
		-\theta(\bar{F},\bar{\eta})\la\nu_{0,x},\li-\bar{\eta}\ra(0,x)\bigg]\:dx \nonumber \\
		&+\int_{0}^{T}\!\!\!\!\int\varphi'(t)\bigg[-\frac{\partial e}{\partial F_{i\alpha}}(\bar{F},\bar{\eta})(F_{i\alpha}-\bar{F}_{i\alpha})-\la\nu_{t,x},\bar{v}_i(\lvi-\bar{v}_i)\ra \nonumber \\
		&\qquad\qquad\qquad
		+\left\la\nu_{t,x},\frac{1}{2}|\lv|^2+e(\lf,\li)-\frac{1}{2}|\bar{v}|^2-e(\bar{F},\bar{\eta})\right\ra \nonumber \\
		&\qquad\qquad\qquad\qquad\qquad\qquad\quad
		-\theta(\bar{F},\bar{\eta})\la\nu_{t,x},\li-\bar{\eta}\ra+\bg\bigg]\:dx\:dt \nonumber \\
		&\!\!\!\!\!\!
		\geq-\int_{0}^{T}\!\!\!\!\int\varphi(t)\bigg[
		-\partial_t\Big(\frac{\partial e}{\partial F_{i\alpha}}(\bar{F},\bar{\eta})\Big)(F_{i\alpha}-\bar{F}_{i\alpha})
		\nonumber \\
		&\qquad\qquad\qquad
		+\partial_{\alpha}\bar{v}_i\left(\left\la\nu_{t,x},\frac{\partial e}{\partial F_{i\alpha}}(\lf,\li)\right\ra-\frac{\partial e}{\partial F_{i\alpha}}(\bar{F},\bar{\eta})\right) \nonumber \\
		&\qquad\qquad\qquad
		-\partial_t\theta(\bar{F},\bar{\eta})(\left\la\nu_{t,x},\li\right\ra-\bar{\eta}) \nonumber \\
		&\qquad\qquad\qquad
		-\theta(\bar{F},\bar{\eta})\left(\left\la\nu_{t,x},\frac{r}{\theta(\lf,\li)}\right\ra-\frac{\bar{r}}{\theta(\bar{F},\bar{\eta})}\right)+\la\nu_{t,x},r-\bar{r}\ra\bigg]\:dx\:dt\nonumber \\
		&\!\!\!\!\!\!
		=:-\int_{0}^{T}\int\varphi(t)\mathcal{R}(t,x)\:dx\:dt\;.
	\end{align} 	
	Using relations \eqref{e:constitutive_functions}, the entropy production identity (\ref{e:entropy.prod}) that holds for strong solutions and equation \eqref{e:adiabatic_thermoelasticity}$_1$, the quantity  $\mathcal{R}(t,x)$ in the integrand on the right hand-side of (\ref{entr.eq.1.mv}) becomes
	\begin{align}
		\label{entr.eq.2.mv}
		\mathcal{R}&=-\partial_t\bar{F}_{j\beta}\frac{\partial^2e}{\partial F_{i\alpha} \partial F_{j\beta}}(\bar{F},\bar{\eta})
		(F_{i\alpha}-\bar{F}_{i\alpha})
		-\partial_t\bar{\eta}\frac{\partial^2e}{\partial F_{i\alpha} \partial \eta}(\bar{F},\bar{\eta})
		(F_{i\alpha}-\bar{F}_{i\alpha})\nonumber\\
		&+\partial_t \bar{F}_{i\alpha}\left\la\nu_{t,x},\frac{\partial e}{\partial F_{i\alpha}}(\lambda_{F},\lambda_{\eta})
		-\frac{\partial e}{\partial F_{i\alpha}}(\bar{F},\bar{\eta})\right\ra
		-\partial_t \bar{F}_{i\alpha}\frac{\partial^2e}{\partial F_{i\alpha} \partial \eta}(\bar{F},\bar{\eta})
		\la\nu_{t,x},\li-\bar{\eta}\ra\nonumber\\
		&-\partial_t\bar{\eta}\frac{\partial^2e}{\partial\eta^2}(\bar{F},\bar{\eta})\la\nu_{t,x},\li-\bar{\eta}\ra\nonumber\\
		&+\partial_t\bar{\eta}\la\nu_{t,x},\theta(\lf,\li)-\theta(\bar{F},\bar{\eta})\ra
		-\partial_t\bar{\eta}\la\nu_{t,x},\theta(\lf,\li)-\theta(\bar{F},\bar{\eta})\ra\nonumber\\
		&-\theta(\bar{F},\bar{\eta})\left(\left\la\nu_{t,x},\frac{r}{\theta(\lf,\li)}\right\ra
		-\frac{\bar{r}}{\theta(\bar{F},\bar{\eta})}\right)+\la\nu_{t,x},r-\bar{r}\ra\nonumber\\
		&\!\!\!\!\!\!=
		\partial_t\bar{\eta}\Bigg[\left\la\nu_{t,x},\theta(\lf,\li)-\theta(\bar{F},\bar{\eta})\right\ra\nonumber\\
		&\qquad\qquad\qquad
		-\frac{\partial\theta}{\partial F_{i\alpha}}(\bar{F},\bar{\eta})(F_{i\alpha}-\bar{F}_{i\alpha})
		-\frac{\partial\theta}{\partial\eta}(\bar{F},\bar{\eta})\la\nu_{t,x},\li-\bar{\eta}\ra\Bigg]\nonumber\\
		&+\partial_t \bar{F}_{j\beta}\Bigg[\left\la\nu_{t,x},\Sigma_{j\beta}(\lf,\li)-\Sigma_{j\beta}(\bar{F},\bar{\eta})\right\ra\nonumber\\
		&\qquad\qquad\qquad
		-\frac{\partial^2e}{\partial F_{i\alpha} \partial F_{j\beta}}(\bar{F},\bar{\eta})(F_{i\alpha}-\bar{F}_{i\alpha})
		-\frac{\partial^2e}{\partial \eta\, \partial F_{j\beta}}(\bar{F},\bar{\eta})\la\nu_{t,x},\li-\bar{\eta}\ra\Bigg]\nonumber\\
		&-\frac{\bar{r}}{\theta(\bar{F},\bar{\eta})}\la\nu_{t,x},\theta(\lf,\li)-\theta(\bar{F},\bar{\eta})\ra\nonumber\\
		&-\theta(\bar{F},\bar{\eta})\left(\left\la\nu_{t,x},\frac{r}{\theta(\lf,\li)}\right\ra
		-\frac{\bar{r}}{\theta(\bar{F},\bar{\eta})}\right)+\la\nu_{t,x},r-\bar{r}\ra\nonumber\\
		&\!\!\!\!\!\!=\partial_t\bar{\eta}\left\la\nu_{t,x},\theta(\lf,\li|\bar{F},\bar{\eta})\right\ra
		+\partial_t\bar{F}_{j\beta}\left\la\nu_{t,x},\Sigma_{j\beta}(\lf,\li|\bar{F},\bar{\eta})\right\ra\nonumber\\
		&+\left\la\nu_{t,x},(\theta(\lf,\li)-\theta(\bar{F},\bar{\eta}))\left(\frac{r}{\theta(\lf,\li)}-\frac{\bar{r}}{\theta(\bar{F},\bar{\eta})}\right)\right\ra.
	\end{align}
	Above, we have used the following notation:
	\begin{align}
		\begin{split}
			\label{partial.theta.rel.mv}
			\left\la\nu_{t,x},\theta(\lf,\li|\bar{F},\bar{\theta})\right\ra
			:=&\Bigg\langle\nu_{t,x},\theta(\lf,\li)-\theta(\bar{F},\bar{\theta})\\
			&-\frac{\partial \theta}{\partial F_{i\alpha}}(\bar{F},\bar{\eta})(F_{i\alpha}-\bar{F}_{i\alpha})
			-\frac{\partial \theta}{\partial\eta}(\bar{F},\bar{\eta})(\li-\bar{\eta})\Bigg\rangle\;,
		\end{split}
	\end{align}
	and
	\begin{align}
		\begin{split}
			\label{partial.s.rel.mv}
			\left\la\nu_{t,x},\Sigma_{i\alpha}(\lf,\li|\bar{F},\bar{\eta})\right\ra
			:=&\Bigg\langle\nu_{t,x},\Sigma_{i\alpha}(\lf,\li)-\Sigma_{i\alpha}(\bar{F},\bar{\eta})\\
			&\!\!\!\!\!\!\!\!\!\!\!\!\!\!\!\!\!\!\!\!
			-\frac{\partial^2 e}{\partial F_{i\alpha} \partial F_{j\beta}}(\bar{F},\bar{\eta})(\lf-\bar{F})
			-\frac{\partial^2 e}{\partial F_{i\alpha} \partial \eta}(\bar{F},\bar{\eta})(\li-\bar{\eta})\Bigg\rangle.
		\end{split}
	\end{align}  
	If we define the averaged quantity
	\begin{align}
		\begin{split}
			\label{def.I}
			I(\lambda_{U}|\bar{U})=I(\lf,\lv,\li|\bar{F},\bar{v},\bar{\eta})
			:=\frac{1}{2}|\lv-\bar{v}|^2+e(\lf,\li|\bar{F},\bar{\eta})
		\end{split}
	\end{align}  
	for
	\begin{align*}
		e(\lf,&\li|\bar{F},\bar{\eta})\\
		&:=e(\lf,\li)-e(\bar{F},\bar{\eta})
		-\frac{\partial e}{\partial F_{i\alpha}}(\bar{F},\bar{\eta})(\lf-\bar{F})_{i\alpha}
		-\frac{\partial e}{\partial \eta}(\bar{F},\bar{\eta})(\li-\bar{\eta}),
	\end{align*}
	we observe that the term on the left hand-side of \eqref{entr.eq.1.mv} becomes
	\begin{align}
		\label{entr.eq.3.mv}
		\begin{split}
			&-\frac{\partial e}{\partial F_{i\alpha}}(\bar{F},\bar{\eta})(F_{i\alpha}-\bar{F}_{i\alpha})-\la\nu_{t,x},\bar{v}_i(\lvi-\bar{v}_i)\ra \\
			&\quad
			+\left\la\nu_{t,x},\frac{1}{2}|\lv|^2+e(\lf,\li)-\frac{1}{2}|\bar{v}|^2-e(\bar{F},\bar{\eta})\right\ra
			-\theta(\bar{F},\bar{\eta})\la\nu_{t,x},\li-\bar{\eta}\ra\\
			&\;\;
			=\left\la\nu_{t,x},\frac{1}{2}|\lv-\bar{v}|^2\right\ra+\left\la\nu_{t,x},e(\lf,\li|\bar{F},\bar{\eta})\right\ra\\
			&\;\;
			=\left\la\nu_{t,x},I(\lf,\lv,\li|\bar{F},\bar{v},\bar{\eta})\right\ra.
		\end{split}
	\end{align}
	We then combine (\ref{entr.eq.1.mv}),(\ref{entr.eq.2.mv}), and (\ref{entr.eq.3.mv}) 
	to arrive at the 
	relative entropy inequality
	\begin{align}
		\label{rel.en.id.mv}
		\begin{split}
			\int&\varphi(0)[\left\la\nu_{0,x},I(\lambda_{U_0}|\bar{U_0})\right\ra\:dx]\\
			&+\int_{0}^{T}\int\varphi'(t)\left[\left\la\nu_{t,x},I(\lambda_{U}|\bar{U})\right\ra\:dx\:dt+\bg(dx\,dt)\right]\\
			&\geq-\int_{0}^{T}\int\varphi(t)\Bigg[\partial_t\bar{\eta}\left\la\nu_{t,x},\theta(\lf,\li|\bar{F},\bar{\eta})\right\ra
			+\partial_t\bar{F}_{j\beta}\left\la\nu_{t,x},\Sigma_{j\beta}(\lf,\li|\bar{F},\bar{\eta})\right\ra\\
			&\qquad\qquad\qquad
			+\left\la\nu_{t,x},(\theta(\lf,\li)-\theta(\bar{F},\bar{\eta}))\left(\frac{r}{\theta(\lf,\li)}-\frac{\bar{r}}{\theta(\bar{F},\bar{\eta})}\right)\right\ra
			\Bigg]\:dx\,dt.
		\end{split}
	\end{align}
	
	\section{Measure-valued versus strong uniqueness} \label{sec_U}
	
	Due to the relative entropy inequality \eqref{rel.en.id.mv}, we may now show that classical solutions are unique in the class of dissipative measure-valued solutions. Henceforth, we assume that the internal energy satisfies the following assumptions:	
	\begin{enumerate}
		\item[$(H_1)$] $e\in C^3(\mathbb{R}^{d\times d}\times\mathbb{R})$
		\item[$(H_2)$] $c(|F|^p+|\eta|^q-1)\leq e(F,\eta)\leq c(|F|^p+|\eta|^q+1)$
		\item[$(H_3)$] 	$|e_F(F,\eta)|\lesssim 1+|F|^{p-1}+|\eta|^{q\frac{p-1}{p}},$ and
		$|e_\eta(F,\eta)|\lesssim 1+|F|^{p\frac{q-1}{q}}+|\eta|^{q-1}.$
	\end{enumerate}
	
	To establish the measure-valued vs strong uniqueness result, we first assert that the following bounds on the relative entropy and the terms on the right hand side of \eqref{rel.en.id.mv} can be obtained given the above hypotheses and the quasiconvexity assumption.
	
	\begin{Lemma}\label{Lemma1}
		Given hypotheses $(H_1)-(H_3)$, for $p,\, q\geq 2$, assume that the 
		smooth solution $(\bar{F},\bar{v},\bar{\eta})$ lies in the compact set
		\begin{equation*}
			\Gamma_{K}:=\left\{(\bar{F},\bar{v},\bar{\eta}): |\bar{F}(t,\cdot)|\leq K, |\bar{v}(t,\cdot)|\leq K, |\bar{\eta}(t,\cdot)|\leq K \right\}
		\end{equation*}
		for a positive constant $K$. Then there exist constants $C_1,C_2,C_3>0$ such that
		\begin{equation}
			\label{bound_rel.1}
			|I(F,v,\eta|\bar{F},\bar{v},\bar{\eta})|\leq C_1 \left(|v-\bar{v}|^2+|V_p(F-\bar{F})|^2+|V_q(\eta-\bar{\eta})|^2\right)
		\end{equation}
		\begin{equation}
			\label{bound_rel.2}
			|\theta(F,\eta|\bar{F},\bar{\eta})|\leq C_2 \left(|V_p(F-\bar{F})|^2+|V_q(\eta-\bar{\eta})|^2\right)
		\end{equation}
		and
		\begin{equation}
			\label{bound_rel.3}
			|\Sigma(F,\eta|\bar{F},\bar{\eta})|\leq C_3 \left(|V_p(F-\bar{F})|^2+|V_q(\eta-\bar{\eta})|^2\right).
		\end{equation}
		Under the additional hypothesis:
		\begin{equation}
			\label{lowerb}
			\theta (F, \eta) = \frac{\partial e}{\partial \eta}(F, \eta) \ge \delta >0 \,  ,
		\end{equation}
		and given that $r(t,x)=\bar r(t,x)\in L^{\infty}(Q_T),$ there exist a constant $C_4>0$ such that
		\begin{align}
			\label{bound_rel.4}
			\begin{split}
				\Bigg|(\theta(F,\eta)-\theta(\bar{F},\bar{\eta}))&\left(\frac{r}{\theta(F,\eta)}-\frac{\bar{r}}{\theta(\bar{F},\bar{\eta})}\right)\Bigg|\\
				&\leq C_4 \left(|V_p(F-\bar{F})|^2+|V_q(\eta-\bar{\eta})|^2\right)
			\end{split}
		\end{align}
		for all $(\bar{F},\bar{v},\bar{\eta})\in\Gamma_{K}$.
	\end{Lemma}
	
	\begin{proof}
		For the proof of \eqref{bound_rel.1}, observe that 
		\begin{align*}
			e(F,\eta&|\bar{F},\bar{\eta})=\\
			&\int_0^1 (1-s) D^2 e(\bar{F}+s(F-\bar{F}),\bar{\eta}+s(\eta-\bar{\eta}))(F-\bar{F},\eta-\bar{\eta}):(F-\bar{F},\eta-\bar{\eta})\:ds.
		\end{align*}
		Given the set of hypotheses $(H_1)-(H_3)$, \eqref{bound_rel.1} follows due to the Appendix as it coincides with the proof of Lemma \ref{growths} (a) by setting $\xi_1=F-\bar{F},$ $\xi_2=\eta-\bar{\eta},$ $\lambda_1=\bar{F},$ $\lambda_2=\bar{\eta}$ and $z_1=z_2=0.$

		Moving to bound \eqref{bound_rel.2}, we cannot use the proof in the Appendix directly, as $\theta$ does not satisfy the same
		growth conditions as $e.$ We start by expressing $\theta(F,\eta|\bar{F},\bar{\eta})$ as follows: 
		\begin{align*}
			\theta(F,\eta|&\bar{F},\bar{\eta})=\\
			&\int_0^1 (1-s) D^2 \theta(\bar{F}+s(F-\bar{F}),\bar{\eta}+s(\eta-\bar{\eta}))(F-\bar{F},\eta-\bar{\eta}):(F-\bar{F},\eta-\bar{\eta})\:ds
		\end{align*}
		so that
		\begin{align*}
			|\theta(F,\eta|\bar{F},\bar{\eta})| \leq C \left(|F-\bar{F}|^2+|\eta-\bar{\eta}|^2\right)
		\end{align*}
		where $C=C(d,\max D^2 \theta)$ in the region $|F-\bar{F}|+|\eta-\bar{\eta}|\leq1$ and $(\bar{F},\bar{\eta})\in\Gamma_{K}$.
		If $|F-\bar{F}|+|\eta-\bar{\eta}|>1$ and $(\bar{F},\bar{\eta})\in\Gamma_{K}$ we have
		\begin{align*}
			|\theta(F,\eta|\bar{F},\bar{\eta})|&\leq |\theta(F,\eta)-\theta(\bar{F},\bar{\eta})| 
			+ \left|\frac{\partial \theta}{\partial F_{i\alpha}}(\bar{F},\bar{\eta})\right| |F_{i\alpha}-\bar{F}_{i\alpha}|
			+ \left|\frac{\partial \theta}{\partial \eta}(\bar{F},\bar{\eta})\right| |\eta-\bar{\eta}|\\
			&\lesssim |\theta(F,\eta)| +  |F-\bar{F}| +  |\eta-\bar{\eta}| + 1\\
			&\lesssim |F|^{p\frac{q-1}{q}}+|\eta|^{q-1}+  |F-\bar{F}| +  |\eta-\bar{\eta}|+ 1\\
			&\lesssim 2^{p\frac{q-1}{q}-1} (|F-\bar{F}|^{p\frac{q-1}{q}}+|\bar{F}|^{p\frac{q-1}{q}}) + 2^{q-2} (|\eta-\bar{\eta}|^{q-1}+|\bar{\eta}|^{q-1})\\
			&\qquad\qquad\qquad\qquad\qquad\qquad\qquad
			+  |F-\bar{F}| +  |\eta-\bar{\eta}| + 1\\
			&\lesssim |F-\bar{F}|^p+|\eta-\bar{\eta}|^q +  |F-\bar{F}| +  |\eta-\bar{\eta}| + 1\\
			&\lesssim |F-\bar{F}|^p +  |\eta-\bar{\eta}|^q +  |F-\bar{F}| + |\eta-\bar{\eta}| \\
			&\lesssim |F-\bar{F}|^p +  |\eta-\bar{\eta}|^q +  (|F-\bar{F}| + |\eta-\bar{\eta}|)^2 \\
			&\lesssim |F-\bar{F}|^p +  |\eta-\bar{\eta}|^q + |F-\bar{F}|^2 + |\eta-\bar{\eta}|^2 \\
			&\leq C \left(|V_p(F-\bar{F})|^2+|V_q(\eta-\bar{\eta})|^2\right),
		\end{align*}
		because of $(H_3)_2,$ Minkowski's inequality and Young's inequality.
		
		Bound \eqref{bound_rel.3} can be obtained similarly by employing $(H_3)_1$, as now the function $\Sigma(F,\eta)$
		is given by \eqref{e:constitutive_functions}$_2$ as a partial derivative of $e(F,\eta)$. 
		
		\textcolor{black}{
			Finally for \eqref{bound_rel.4} in the region where $|F-\bar{F}|+|\eta-\bar{\eta}|\leq1$ and $(\bar{F},\bar{\eta})\in\Gamma_{K}$ we have
			\begin{align*}
				\Bigg|(\theta(F,\eta)-\theta(\bar{F},\bar{\eta}))\left(\frac{r}{\theta(F,\eta)}-\frac{\bar{r}}{\theta(\bar{F},\bar{\eta})}\right)\Bigg| &\lesssim
				\frac{|\theta(F,\eta)-\theta(\bar{F},\bar{\eta})|^2}{\theta(F,\eta)\theta(\bar{F},\bar{\eta})}\\
				&\lesssim |\theta(F,\eta)-\theta(\bar{F},\bar{\eta})|^2\\
				&\lesssim |F-\bar{F}|^2+|\eta-\bar{\eta}|^2,
			\end{align*}
			where the constants involved depend on $(||r||_{L^{\infty}},K).$ If $|F-\bar{F}|+|\eta-\bar{\eta}|>1,$ given \eqref{lowerb} it holds that
			\begin{align*}
				\Bigg|(\theta(F,\eta)-\theta(\bar{F},\bar{\eta}))&\left(\frac{r}{\theta(F,\eta)}-\frac{\bar{r}}{\theta(\bar{F},\bar{\eta})}\right)\Bigg| \lesssim |\theta(F,\eta)-\theta(\bar{F},\bar{\eta})|\lesssim |\theta(F,\eta)| + 1,
			\end{align*}
			for all $(\bar{F},\bar{\eta})\in\Gamma_{K}$ and 
			for a constant depending on $(||r||_{L^{\infty}},\delta,K).$ Then we can proceed as above to bound the term $|\theta(F,\eta)|.$ This concludes the proof.}
	\end{proof}
	
	In Section 6, Theorem \ref{theorem:3Garding}, we prove the following G\aa rding-type inequality 
	\begin{align}
		\label{thm2.Garding}
		\int |V_p(\nabla\phi)|^2+|V_q(\psi)|^2 \leq C_0\int e(\baf+\nabla\phi,\bai+\psi|\baf,\bai) + C_1\int |V_p(\phi)|^2,
	\end{align}
	where $ C_0= C_0(e,K)>0$, $ C_1= C_1(e,K)>0$ and $\psi\in L^q(Q)$ with $\int\psi=0$ and $\phi\in W^{1,p}_0(Q).$ 
	We will use \eqref{thm2.Garding} for $(\baf,\bai)$ a classical (Lipschitz) solution to \eqref{e:adiabatic_thermoelasticity} to prove the measure-valued versus strong uniqueness result. 
	Hence, the proof of Theorem \ref{Uniqueness_quasiconvex},relies on Theorems \ref{thm:main} and \ref{theorem:3Garding},  and we 
	refer the reader to those statements for the proof of \eqref{thm2.Garding}.
	
	\begin{Lemma}\label{lemma2}
		Suppose that $(\nu,\bg,F,v,\eta)$ is a dissipative measure-valued solution to adiabatic thermoelasticity according to Definition \ref{MVS_adiabatic_theromelasticity} and that $(\bar{F},\bar{v},\bar{\eta})$ is a classical solution to \eqref{e:adiabatic_thermoelasticity}
		with initial data ($F^0,v^0,\eta^0)$ and $(\bar{F}^0,\bar{v}^0,\bar{\eta}^0)$ respectively. Under the assumptions of Theorem $3$ and by
		denoting $\nu_0=\nu_{t_0,x}$,  it holds that
		\begin{align}
			\label{lemma2.bound1}
			\begin{split}
				\int \la \nu_0, |V_p(\lf-\bar{F}(t_0,x))|^2&+|V_q(\li-\bar{\eta}(t_0,x))|^2 \ra \:dx\\
				&\leq\tilde C_0\int \la \nu_0,e(\lf,\li|\bar{F}(t_0,x),\bar{\eta}(t_0,x))\ra \:dx\\
				&\qquad\qquad
				+\tilde C_1\int|V_p(y(t_0,x)-\bar y(t_0,x))|^2  \:dx,
			\end{split}
		\end{align}
		and 
		\begin{align}
			\label{lemma2.bound2}
			\begin{split}
				\int \la \nu_0,& |V_p(\lf-\bar{F}(t_0,x))|^2+|V_q(\li-\bar{\eta}(t_0,x))|^2 + |\lv-\bar{v}(t_0,x)|^2\ra \:dx\\
				&\leq\tilde C_0\int \la \nu_0,I(F,v,\eta|\bar{F},\bar{v},\bar{\eta};(t_0,x))\ra \:dx\\
				&\qquad\qquad
				+\tilde C_1\int|V_p(y(t_0,x)-\bar y(t_0,x))|^2  \:dx,
			\end{split}
		\end{align}
		for almost all $t_0\in(0,T).$ In addition, at $t=0$ 
		\begin{align}
			\label{lemma2.bound3}
			\int I(F^0,v^0,\eta^0|\bar{F}^0,\bar{v}^0,\bar{\eta}^0) \: 
			\leq\tilde c \int |v^0-\bar{v}^0|^2 +  |V_p(F^0-\bar{F}^0)|^2+|V_q(\eta^0-\bar{\eta}^0)|^2\:.
		\end{align}
	\end{Lemma}
	
	\begin{proof}
		We prove \eqref{lemma2.bound1} as a result of Theorems  \ref{Theorem:1localisation} and \ref{theorem:3Garding}.
		In \eqref{thm2.Garding} take $\bar F=\bar F(t_0,\cdot)$ and $\bar \eta=\bar \eta(t_0,\cdot)$ and then choose $\phi=z_k(t,\cdot)-\bar 
		y(t_0,\cdot)$ and $\psi=w_k(t,\cdot)-\bar\eta(t_0,\cdot),$ as in Theorem $1.$ Observe in this case, that the relative quantity 
		$e(\baf+\nabla\phi,\bai+\psi|\baf,\bai)$ becomes $e(\nabla z_k,w_k|\bar{F}(t_0,x),\bar{\eta}(t_0,x)).$ Integrating the resulting inequality in 
		time and since $(\nabla z_k,w_k)$ generates the measure $(\nu_{t_0,x})_{x\in Q},$ $(|\nabla z_k|^p+|w_k|^q)$ is weakly relatively compact 
		in $L^1(Q_T)$ and $z_k\to y(t_0,\cdot)$ strongly in $L^p(Q)$, we obtain in the limit as $k\to\infty,$ inequality \eqref{lemma2.bound1}.
		
		For \eqref{lemma2.bound2}, we exploit inequality \eqref{lemma2.bound1}, together with the fact that the relative entropy $I$ is given 
		as a sum in \eqref{def.I}.  Let $(F^k,v^k,\eta^k)$ be a generating sequense satisfying 
		\begin{equation*}
			F^{k}\in L^{\infty}(0,T;L^p(Q)) ,\quad v^{k}\in L^{\infty}(0,T;L^2(Q)), \quad \eta^{k}\in L^{\infty}(0,T;L^q(Q)),
		\end{equation*}
		and consider a function $g:Q_T\times\mathbb{R}^{d\times d}\times\mathbb{R}^{d}\times\mathbb{R}$ such that
		$$g(t,x,F,v,\eta)=g_{v}(t,x,v)+g_{F,\eta}(t,x,F,\eta),$$
		where 
		\begin{align*}
			|g_{v}|\leq c (1+|v|^2)  \quad \text{and} \quad  |g_{F,\eta}|\leq c (1+|F|^p+|\eta|^q).
		\end{align*}
		The action of the generated measure $\nu_{t,x}$ is equivalent with the action of  $\nu^v \otimes \nu^{F,\eta}$ for the measures
		$\nu^v$ and $\nu^{F,\eta}$ are generated by the sequenses $(v^k)$ and $(F^k,\eta^k)$ respectively. Therefore, it suffices to add the 
		term
		$$\int\la \nu_0^v,|v-\bar{v}(t_0,x))|^2\ra \:dx$$
		to inequality \eqref{thm2.Garding}. Finally, \eqref{lemma2.bound3} follows directly from Lemma \ref{Lemma1} and in particular bound 
		\eqref{bound_rel.1}.
	\end{proof}
	
	Using the averaged relative entropy inequality \eqref{rel.en.id.mv}, we are now in a position to prove that in the presence of a classical 
	solution, given that the associated Young measure is initially a Dirac mass, the dissipative measure-valued solution must coincide with the 
	classical one.
	
	\begin{Theorem}\label{Uniqueness_quasiconvex}
		Let $\bar{U}$ be a Lipschitz bounded solution of (\ref{e:adiabatic_thermoelasticity}),(\ref{e:entropy.prod}) with initial data $\bar{U}^0$ and 
		$(\nu_{t,x},\bg,U)$ be a dissipative measure-valued solution  satisfying (\ref{e:mv.sol.xi-v-eta}),(\ref{e:mv.sol.Fvt.energy}), with initial data 
		also $\bar{U}^0,$ both under the constitutive assumptions (\ref{e:constitutive_functions}) and such that 
		$r(t,x)=\bar{r}(t,x)\in L^{\infty}(Q_T)$. Suppose that $e$  is quasiconvex according to Definition \ref{defqcc} and the hypotheses 
		$(H_1)-(H_3)$ hold for $p,\,q\geq 2$, together with \eqref{lowerb}. If $\nu_{0,x}=\delta_{\bar{U}^0(x)}$ and $\bg_0=0,$  we have that 
		$\nu_{t,x}=\delta_{\bar{U}}$ and $U=\bar U$ a.e. on $Q_T.$
	\end{Theorem}
	\begin{proof}
		Let $\{\varphi_n\}$ be a sequence of monotone decreasing functions such that $\varphi_n\geq 0,$ for all $n\in\mathbb{N},$ converging as $n \to \infty$ to the Lipschitz function
		\begin{align*}
			\varphi(\tau)=\begin{dcases}
				1 & 0\leq\tau\leq t\\
				\frac{t-\tau}{\varepsilon}+1 & t\leq\tau\leq t+\varepsilon\\
				0 & \tau\geq t+\varepsilon
			\end{dcases}
		\end{align*}
		for some $\varepsilon>0.$ Writing the relative entropy inequality~\eqref{rel.en.id.mv} for $r(t,x)=\bar{r}(t,x),$ tested against the functions $\varphi_n$ we have 
		\begin{align*}
			\int&\varphi_n(0)[\left\la\nu_{0,x},I(\lambda_{U_0}|\bar{U_0})\right\ra\:dx]\\
			&+\int_{0}^{T}\int\varphi_n'(t)\left[\left\la\nu_{t,x},I(\lambda_{U}|\bar{U})\right\ra\:dx\:dt+\bg(dx\,dt)\right]\\
			&\geq-\int_{0}^{T}\int\varphi_n(t)\Bigg[\partial_t\bar{\eta}\left\la\nu_{t,x},\theta(\lf,\li|\bar{F},\bar{\eta})\right\ra
			+\partial_t\bar{F}_{j\beta}\left\la\nu_{t,x},\Sigma_{j\beta}(\lf,\li|\bar{F},\bar{\eta})\right\ra\\
			&\qquad\qquad\qquad\qquad
			+\left\la\nu_{t,x},(\theta(\lf,\li)-\theta(\bar{F},\bar{\eta}))\left(\frac{\bar r}{\theta(\lf,\li)}-\frac{\bar{r}}{\theta(\bar{F},\bar{\eta})}\right)\right\ra
			\Bigg]\:dx\,dt.
		\end{align*}
		Passing to the limit as $n\to \infty$ we get
		\begin{align*}
			\int &\left\la\nu_{0,x},I(\lf,\lv,\li|\bar{F},\bar{v},\bar{\eta})\right\ra(0,x)\:dx\\
			&\quad
			-\frac{1}{\varepsilon}\int_{t}^{t+\varepsilon}
			\int \left[\left\la\nu_{t,x},I(\lf,\lv,\li|\bar{F},\bar{v},\bar{\eta})\right\ra\:dx\:d\tau+\bg(dxd\tau)\right]  \\ 
			&\geq- \int_0^{t+\varepsilon}\!\!\!\int \!\Big[\partial_t\bar{\eta}\left\la\nu_{t,x},\theta(\lf,\li|\bar{F},\bar{\eta})\right\ra
			+\partial_t\bar{F}_{j\beta}\left\la\nu_{t,x},\Sigma_{j\beta}(\lf,\li|\bar{F},\bar{\eta})\right\ra\\
			&\qquad\qquad\qquad\quad
			+\left\la\nu_{t,x},(\theta(\lf,\li)-\theta(\bar{F},\bar{\eta}))\left(\frac{\bar r}{\theta(\lf,\li)}-\frac{\bar{r}}{\theta(\bar{F},\bar{\eta})}\right)\right\ra\Big]\:dx\,d\tau,
		\end{align*}
		and using the estimates \eqref{bound_rel.2},  \eqref{bound_rel.3}, and  \eqref{bound_rel.4} we arrive at
		\begin{align*}
			\int& \left\la\nu_{0,x},I(\lf,\lv,\li|\bar{F},\bar{v},\bar{\eta})\right\ra(0,x)\:dx\\
			&\quad-\frac{1}{\varepsilon}\int_{t}^{t+\varepsilon}
			\int \left[\left\la\nu_{t,x},I(\lf,\lv,\li|\bar{F},\bar{v},\bar{\eta})\right\ra\:dx\:d\tau+\bg(dxd\tau)\right] \\
			&\geq- C \int_0^{t+\varepsilon}\!\!\!\int \left\la\nu_{t,x},|V_p(\lf-\bar{F})|^2+|V_q(\li-\bar{\eta})|^2\right\ra\:dx\,d\tau.
		\end{align*}
		Passing now to the limit as $\varepsilon \to 0^{+}$ and using the fact that $\bg\geq 0$ we get 
		\begin{align}\label{thm.estimate0}
			\int [\la\nu_{t,x},I(\lf,\lv,&\li|\bar{F},\bar{v},\bar{\eta})\ra\:dx
			\leq \int \left\la\nu_{0,x},I(\lf,\lv,\li|\bar{F},\bar{v},\bar{\eta})\right\ra(0,x)\:dx\nonumber\\
			&+C \int_0^t \int \left\la\nu_{t,x},|V_p(\lf-\bar{F})|^2+|V_q(\li-\bar{\eta})|^2\right\ra\:dx\,d\tau,
		\end{align}
		which together with \eqref{bound_rel.1}, \eqref{lemma2.bound2}, and \eqref{lemma2.bound3} yields
		\begin{align}
			\label{thm.estimate}
			\begin{split}
				\int \la\nu_{t,x},|\lv-\bar{v}|^2+&|V_p(\lf-\bar{F})|^2+|V_q(\li-\bar{\eta})|^2\ra\:dx\\ 
				&\leq C \int_0^t \int \left\la\nu_{t,x},|V_p(\lf-\bar{F})|^2+|V_q(\li-\bar{\eta})|^2\right\ra\:dx\,d\tau\\
				&\qquad\qquad
				+ C \int |V_p(y-\bar y)|^2 \:dx,
			\end{split}
		\end{align}
		for a.e. $t\in (0,T).$ Here we used the assumptions that the two solutions have the same initial data and that $\bg_0=0.$ Note that the constant $C$ depends only on the smooth bounded solution $\bar{U}.$ 
		To apply Gr\"onwall's inequality and close our argument it remains to estimate the last term on the right hand-side of \eqref{thm.estimate}.
		This was done in \cite{KS2017}, using elliptic estimates and equation \eqref{e:adiabatic_thermoelasticity}$_1$, together with
		\eqref{constraint}. We also refer the reader to \cite[Remark 6]{KV2020} for the derivation of such lower order estimates in a more 
		general setting. In particular, we have
		\begin{align*}
			\int |V_p(y-\bar y)|^2 \:dx\leq C \int_0^t \int &\left\la\nu_{t,x},|V_p(\lf-\bar{F})|^2+|\lv-\bar{v}|^2\right\ra\:dx\,d\tau\\
			&\qquad\qquad
			+ C \int_0^t \int V_p(y-\bar y)|^2\:dx\,d\tau
		\end{align*}
		and adding the term $\int |V_p(y-\bar y)|^2 \:dx$
		on both sides of \eqref{thm.estimate} we arrive at 
		\begin{align*} 
			\int \big(\la\nu_{t,x},&|\lv-\bar{v}|^2+|V_p(\lf-\bar{F})|^2+|V_q(\li-\bar{\eta})|^2\ra + |V_p(y-\bar y)|^2\big) \:dx\\ 
			&\leq C \int_0^t \int \big(\left\la\nu_{t,x},|V_p(\lf-\bar{F})|^2+|V_q(\li-\bar{\eta})|^2\right\ra + |V_p(y-\bar y)|^2\big) \:dx\,d\tau.
		\end{align*}
		Gr\"onwall's inequality completes the proof.
	\end{proof}
	\begin{Remark}
		The radiative heat supply $r(t,x)$ is a field that can be regulated externally. Therefore, one could think instead the theory of
		thermoelasticity with zero radiative heat supply and prove Theorem \ref{Uniqueness_quasiconvex} in this less general setting.
		In the case $r(t,x)=\bar{r}(t,x)=0,$ the result of Theorem \ref{Uniqueness_quasiconvex} holds without the assumption on the 
		temperature \eqref{lowerb},	and bound \eqref{bound_rel.4}.
	\end{Remark}	
	
	\section{Localisation in time}
	
	In Theorem \ref{Uniqueness_quasiconvex} we are required to localise our measure-valued solution in time and the generating sequences for these localised measures must be given by a proper time modification of the generating sequence for $\nu$. However, due to the lack of equiintegrability of the assumed generating sequence, we need to construct a new sequence which lies on the desired wave cone, has suitable equiintegrability and convergence properties and generates the localised measure $(\nu_{t_0,x})_{x\in Q}$. To this end, we present the time-dependent decomposition lemma below which is not needed if, instead of measure-valued solutions, weak solutions are considered.
	\begin{Theorem}\label{Theorem:1localisation}
		Let $\nu=(\nu_{t,x})_{(t,x)\in Q_T}$ be a family of probability measures generated by a sequence $(\nabla y_k,\eta_k)$ such that 
		\begin{align*}
			&(y_k)\hspace{0.2cm}\text{is bounded in}\hspace{0.2cm}L^\infty(0,T;W^{1,p}(Q))\\
			&(\partial_t\nabla y_k) \hspace{0.2cm}\text{is bounded in}\hspace{0.2cm}L^\infty(0,T;H^{-1}(Q))\\
			&(\eta_k)\hspace{0.2cm}\text{is bounded in}\hspace{0.2cm}L^\infty(0,T;L^{q}(Q)),
		\end{align*}
		and write $(\nabla y,\eta)=\langle\nu,\rm{id}\rangle$ for its centre of mass. Then, for almost all $t_0\in(0,T)$, there exists a sequence $(\nabla z_k,w_k)$ also bounded in $L^\infty(0,T;L^{p}(Q))\times L^\infty(0,T;L^{q}(Q))$ with the following properties 
		\begin{enumerate}
			\item[(1)] $(\nabla z_k,w_k)$ generates the measure $(\nu_{t_0,x})_{x\in Q}$ as a p-q-Young measure;
			\item[(2)] $(|\nabla z_k|^p+|w_k|^q)$ is weakly relatively compact in $L^1(Q_T)$;
			\item[(3)] $z_k\to y(t_0,\cdot)$ strongly in $L^p(Q_T)$.
		\end{enumerate}
		\begin{proof}
			For $t_0\in (0,T)$ define
			\begin{align*}
				y^{k,\eps}(t,x):=y_k(t_0+\eps t/T,x),\,\,\,\eta^{k,\eps}(t,x):=\eta_k(t_0+\eps t/T,x).
			\end{align*}
			We claim that for a.e. $t_0$ an appropriate subsequence of $(\eps_k)$ can be chosen such that $(\nabla y^{k,\eps},\eta^{k,\eps})$ generates the measure $(\nu_{t_0,x})_{x\in Q}$ and that $y^{k,\eps}\to y(t_0,\cdot)$ in $L^p(Q_T)$. To this end, note that, up to a subsequence which is not relabelled, for any $g\in C_{p,q}(\mathbb{R}^{d\times d}\times\mathbb{R})$ and any Borel set $E\subseteq Q_T$ for a.e. $t_0\in (0,T)$ it holds that 
			\begin{align}\label{eq:loc1}
				\lim_{\eps\to 0}\int_E|\langle\nu_{t_0+\eps t/T,x},g(\lambda_F,\lambda_\eta)\rangle-\langle\nu_{t_0,x},g(\lambda_F,\lambda_\eta)\rangle|=0.
			\end{align}
			This is a consequence of [\cite{KS2017}, Lemma 15] noting that the function $v(t,x)=\langle\nu_{t,x},g\rangle$ is an element of $L^\infty(0,T;L^{1}(Q))$ since 
			\[
			\sup_t\int_Q \langle\nu_{t,x},|\lambda_F|^p+|\lambda_\eta|^q\rangle<\infty.
			\]
			Hence, it follows that for any such $g$ and $E$, denoting by $\mathcal{X}_E$ the characteristic function of $E$ and $t_0$ fixed a.e. in $(0,T)$ using \eqref{eq:loc1}, we infer that 
			\begin{align}
				&\lim_{\eps\to 0}\lim_{k\to\infty}\int_E g(\nabla y^{k,\eps}(t,x),\eta^{k,\eps}(t,x))\nonumber\\
				&= \lim_{\eps\to 0}\lim_{k\to\infty}\frac{T}{\eps}\int_{t_0}^{t_0+\eps}\int_Q \mathcal{X}_E\left((t-t_0)T/\eps,x\right)g(\nabla y^{k,\eps}(t,x),\eta^{k,\eps}(t,x))\nonumber\\
				&=\lim_{\eps\to 0}\int_{Q_T}\mathcal{X}_E(t,x)\langle\nu_{t_0+\eps t/T,x},g(\lambda_F,\lambda_\eta)\rangle\nonumber\\
				&=\int_E\langle\nu_{t_0,x},g(\lambda_F,\lambda_\eta)\rangle.
			\end{align}
			In addition, similarly with [\cite{KS2017},Lemma 16] we have that 
			\begin{align}\label{eq:loc2}
				\lim_{\eps\to 0}\lim_{k\to\infty}\int_{Q_T}|y_k(t_0+\eps t/T,x)-y(t_0,x)|^p=0.
			\end{align}
			Now, for $g$ and $E$ in a countable dense subset of $C_{p,q}(\mathbb{R}^{d\times d}\times\mathbb{R})$ and of the collection of Borel subsets of $Q_T$, respectively, we may choose a subsequence $(\eps_k)$ such that \eqref{eq:loc1} and \eqref{eq:loc2} hold. In particular, for $t_0$ fixed almost everywhere in $(0,T)$,
			\begin{align*}
				&\lim_{k\to\infty}\int_E g\left(\nabla y^{k,\eps_k},\eta^{k,\eps_k}\right)=\int_E\langle\nu_{t_0,x},g(\lambda_F,\lambda_\eta)\rangle,
			\end{align*}
			for all the elements of the countable subsets where $g$ and $E$ belong and, by density, for all $g\in C_{p,q}(\mathbb{R}^{d\times d}\times\mathbb{R})$ and all $E\subseteq Q_T$, i.e.
			\begin{align*}
				g\left(\nabla y^{k,\eps_k},\eta^{k,\eps_k}\right)\rightharpoonup \langle\nu_{t_0,x},g(\lambda_F,\lambda_\eta)\rangle\,\,\,\text{in}\,\,\,L^1(Q_T),
			\end{align*}
			and $\left(\nabla y^{k,\eps_k},\eta^{k,\eps_k}\right)$ generates the measure $(\nu_{t_0,x})_x$. Note also that
			\begin{align*}
				(\nabla y^{k,\eps_k})\subseteq L^\infty(0,T;L^{p}(Q)) \,\,\,\,\text{and}\,\,\,\,(\eta^{k,\eps_k})\subseteq L^\infty(0,T;L^{q}(Q)).
			\end{align*}
			For $n\in\mathbb{N}$ and $(z_1,z_2)\in\mathbb{R}^{d\times d}\times\mathbb{R}$ consider the truncation operator 
			\begin{align*}
				\tau_n(z_1,z_2):=
				\begin{cases}
					(z_1,z_2), &|z_1|^2+|z_2|^2\leq n^2, \\
					n\,{(z_1,z_2)}/{|(z_1,z_2)|}, &|z_1|^2+|z_2|^2>n^2.
				\end{cases} 
			\end{align*}
			We observe that $\tau_n(z_1,z_2)=\left(\tau^F_n(z_1,z_2),\tau^\eta_n(z_1,z_2)\right)$ where 
			\begin{align*}
				\tau_n^F(z_1,z_2):=
				\begin{cases}
					z_1, &|z_1|^2+|z_2|^2\leq n^2, \\
					n\,{z_1}/{|(z_1,z_2)|}, &|z_1|^2+|z_2|^2>n^2,
				\end{cases} 
			\end{align*}
			and $\tau_n^\eta(z_1,z_2)$ is defined respectively. It is straightforward to see that for fixed $n\in\mathbb{N}$ the sequence $\left(|\tau_n^F(z_1,z_2)|^p+|\tau_n^\eta(z_1,z_2)|^q\right)$ is equiintegrable and so 
			\begin{align}\label{eq:loc3}
				&\lim_{n\to\infty}\lim_{k\to\infty}\int_{Q_T}|\tau_n^F(\nabla y^{k,\eps_k},\eta^{k,\eps_k})|^p+|\tau_n^\eta(\nabla y^{k,\eps_k},\eta^{k,\eps_k})|^q \nonumber \\
				&= \lim_{n\to\infty}\int_{Q_T}\langle\nu_{t_0,x},|\tau_n^F(\lambda_F\lambda_\eta)|^p+|\tau_n^\eta(\lambda_F,\lambda_\eta)|^q\rangle
				=\int_{Q_T}\langle\nu_{t_0,x},|\lambda_F|^p+|\lambda_\eta|^q\rangle,
			\end{align}
			where the second equality uses monotone convergence. Moreover, 
			\begin{align}\label{eq:loc4}
				\lim_{n\to\infty}\lim_{k\to\infty}\int_{Q_T}|\tau_n(\nabla y^{k,\eps_k},\eta^{k,\eps_k})-(\nabla y^{k,\eps_k},\eta^{k,\eps_k})|=0
			\end{align}
			due to the $L^1$-equiintegrability of $(\nabla y^{k,\eps_k},\eta^{k,\eps_k})$. Then, from \eqref{eq:loc3} and \eqref{eq:loc4}, there exists a subsequence $k_n$, such that 
			\begin{align*}
				&V_n:=\left(\tau_n^F(\nabla y^{k_n,\eps_{k_n}},\eta^{k_n,\eps_{k_n}}),\tau_n^\eta(\nabla y^{k_n,\eps_{k_n}},\eta^{k_n,\eps_{k_n}})\right)\overset{Y}\to(\nu_{t_0,x})_{x\in Q};\\
				&\left(|\tau_n^F(\nabla y^{k_n,\eps_{k_n}},\eta^{k_n,\eps_{k_n}})|^p+|\tau_n^\eta(\nabla y^{k_n,\eps_{k_n}},\eta^{k_n,\eps_{k_n}})|^q \right)\,\,\,\,\text{is equiinegrable}.
			\end{align*}
			Next, for almost all t, consider the decomposition 
			\begin{align}
				\tilde V_n:=\left(\mathcal{P}_{\rm{curl}}\left(V_n^{(1)}-\int_QV_n^{(1)}\right),\,V_n^{(2)}-\int_QV_n^{(2)}\right),
			\end{align}
			where 
			\begin{align*}
				\tilde V_n^{(1)}&:=\mathcal{P}_{\rm{curl}}\left(V_n^{(1)}-\int_QV_n^{(1)}\right),\,\,\,\,\,\text{}\,\,\,\,\,V_n^{(1)}:=\tau_n^F(\nabla y^{k_n,\eps_{k_n}},\eta^{k_n,\eps_{k_n}}),\\
				\tilde V_n^{(2)}&:=V_n^{(2)}-\int_QV_n^{(2)},\quad\quad\quad\quad\,\,\,\,\,\text{}\,\,\,\,\,V_n^{(2)}:=\tau_n^\eta(\nabla y^{k_n,\eps_{k_n}},\eta^{k_n,\eps_{k_n}})
			\end{align*}
			and $\mathcal{P}_{\rm{curl}}$ denotes the projection operator onto curl-free vector fields. For convenience, let us write $y^n:=y^{k_n,\eps_{k_n}}$ and $\eta^n:=\eta^{k_n,\eps_{k_n}}$ and recall that $\mathcal{P}_{\rm{curl}}$ is a strong $(r,r)$ operator, $1<r<\infty$. Then, for a.e. $t\in(0,T)$,
			\begin{align*}
				\|\tilde V_n^{(1)}(t,\cdot)\|_{L^p(Q)}&\leq C\|\nabla y^n(t,\cdot)\|_{L^p(Q)}\leq C \sup_t \|\nabla y^n(t,\cdot)\|_{L^p(Q)},\\
				\|\tilde V_n^{(2)}(t,\cdot)\|_{L^q(Q)}&\leq C\|\eta^n(t,\cdot)\|_{L^q(Q)}\leq C \sup_t \|\eta^n(t,\cdot)\|_{L^q(Q)},
			\end{align*}
			which shows that the sequences $(\tilde V_n^{(1)})$ and $(\tilde V_n^{(2)})$ are bounded in $L^\infty(0,T;L^{p}(Q))$ and $L^\infty(0,T;L^{q}(Q))$ respectively. To see that $(\tilde V_n)$ generates the measure $(\nu_{t_0,x})_{x\in Q}$, note that, denoting by $\mathcal{P}_{\rm div}$ the projection onto divergence-free vector fields,
			\begin{align}\label{eq:loc5}
				|\left(\nabla y^n(t,\cdot),\eta^n(t,\cdot)\right)-\tilde V_n(t,\cdot)|\leq |\nabla y^n(t,\cdot)-\tilde V_n^{(1)}(t,\cdot)|+|\eta^n(t,\cdot)-\tilde V_n^{(2)}(t,\cdot)|\nonumber\\
				= |\nabla y^n(t,\cdot)-\tau_n^F(\nabla y^n(t,\cdot),\eta^n(t,\cdot))+\mathcal{P}_{\rm{div}}\left(\tau_n^F(\nabla y^n(t,\cdot),\eta^n(t,\cdot))-\nabla y^n(t,\cdot)\right)|\nonumber\\
				+ \Big|\eta^n(t,\cdot)-\tau_n^\eta(\nabla y^n(t,\cdot),\eta^{n}(t,\cdot))+\int_Q \tau_n^\eta(\nabla y^n(t,\cdot),\eta^{n}(t,\cdot))\Big|\nonumber\\
				\leq |\nabla y^n(t,\cdot)-\tau_n^F(\nabla y^n(t,\cdot),\eta^n(t,\cdot))|+|\mathcal{P}_{\rm{div}}\left(\tau_n^F(\nabla y^n(t,\cdot),\eta^n(t,\cdot))-\nabla y^n(t,\cdot)\right)|\nonumber\\
				+|\eta^n(t,\cdot)-\tau_n^\eta(\nabla y^n(t,\cdot),\eta^{n}(t,\cdot))|+\Big|\int_Q \tau_n^\eta(\nabla y^n(t,\cdot),\eta^{n}(t,\cdot))\Big|=:\sum_{i=1}^4I^n_i.
			\end{align}
			However,  for any $\eps>0$ and almost all t
			\begin{align*}
				\mathcal{L}^d\big(\{|\left(\nabla y^n(t,\cdot),\eta^n(t,\cdot)\right)-\tilde V_n(t,\cdot)|>\eps\}&\big)\leq \frac{1}{\eps}\int_{\{|\left(\nabla y^n(t,\cdot),\eta^n(t,\cdot)\right)-\tilde V_n(t,\cdot)|>\eps\}}\eps dx\\
				\leq \frac{1}{\eps}\int_Q|\left(\nabla y^n(t,\cdot),\eta^n(t,\cdot)\right)-\tilde V_n(t,\cdot)|
				&\overset{\eqref{eq:loc5}}\leq \frac{1}{\eps}\int_Q I_1^n+I_2^n+I_3^n+I_4^n.
			\end{align*}
			Then, we claim that
			\begin{align*}
				\mathcal{L}^{d+1}\big(\{\,|\left(\nabla y^n,\eta^n\right)-\tilde V_n|>\eps\}\big)&=\int_0^T\mathcal{L}^d\big(\{|\left(\nabla y^n(t,\cdot),\eta^n(t,\cdot)\right)-\tilde V_n(t,\cdot)|>\eps\}\big)\\
				&\leq\frac{C}{\eps}\int_{Q_T}I_1^n+I_2^n+I_3^n+I_4^n\to 0.
			\end{align*}
			Indeed, for the first term
			\begin{align*}
				&\int_{Q_T}I_1^n=\|\nabla y^n-\tau_n^F(\nabla y^n,\eta^n)\|_{L^1(Q_T)}\leq 2\int_{\{|\nabla y^n|^2+|\eta^n|^2> n^2\}}|\nabla y^n| \\
				&\leq 2 \int_{\{|\nabla y^n|^2+|\eta^n|^2> n^2\}} \frac{\left(|\nabla y^n|+|\eta^n|\right)^2}{|\nabla y^n|+|\eta^n|}\leq 2 \int_{\{|\nabla y^n|+|\eta^n|> n\}} \frac{\left(|\nabla y^n|+|\eta^n|\right)^2}{|\nabla y^n|+|\eta^n|}\\
				&\leq \frac{2}{n}\int_{Q_T}|\nabla y^n|^2+|\eta^n|^2 \leq \frac{2T}{n}\sup_t\int_{Q}|\nabla y^n|^2+|\eta^n|^2 \to 0,
			\end{align*}     
			whenever $n\to\infty$. Similarly we may prove that $\int_{Q_T}I_3^n\to 0$ as $n\to\infty$. Since $\mathcal{P}_{\rm{div}}$ is a weak $(1,1)$ operator, term $\int_{Q_T}I_2^n$ behaves like $\int_{Q_T}I_1^n$ and thus
			\[
			\int_{Q_T}I_2^n\to 0,\,\,\,\,\,\text{whenever}\,\,\,\,\,n\to\infty.
			\]
			Concerning the last term, since $\|\eta^n-\tau_n^\eta(\nabla y^n,\eta^n)\|_{L^1(Q_T)}\to 0$ and $\int_Q\eta^n(t,\cdot)dx=0$ for a.e. $t\in (0,T)$, we infer that 
			\begin{align*}
				\lim_n\int_0^T\Big|\int_Q\tau_n^\eta(\nabla y^n(t,\cdot),\eta^n(t,\cdot))\Big|dt=0,
			\end{align*}
			which concludes the proof of the claim. From the above estimates we  infer that the sequence $\tilde{V}_n:=(\tilde{V}^{(1)}_n,\tilde{V}^{(2)}_n)$ generates the Young measure $(\nu_{t_0,x})_{x\in Q}$.
			\par 
			The equiintegrability of the sequences $\tilde V_n^{(1)}$ and $\tilde V_n^{(2)}$ comes directly from the equiintegrability of  $ V_n^{(1)}$ and $ V_n^{(2)}$. We note that since $\tilde{V}^{(1)}_n(t,\cdot)$ is curl-free for a.e. $t\in(0,T)$, there exists $z_k\in L^\infty(0,T;W^{1,p}(Q))$ s.t. $\tilde{V}^{(1)}_n(t,\cdot)=\nabla z^n(t,\cdot)$. In addition, we set $w_n:=\tilde{V}^{(2)}_n(t,\cdot)$ to serve the requirements of the theorem.
			\par 
			Finally, for the strong convergence of the primitives of the sequence $\tilde V_n^{(1)}$ we follow the proof of [\cite{KS2017},Lemma 16]. We note that by \eqref{eq:loc4} and the Lebesgue interpolation theorem, it holds that
			\begin{align}\label{eq:loc7}
				\lim_n\int_0^1\int_Q|V_n^{(1)}-\nabla y^n|=0\Rightarrow \lim_n\|V_n^{(1)}-\nabla y^n\|_{L^r(L^m)}=0,
			\end{align}
			for all $r<\infty$, $m<p$. Since $V_n^{(1)}:=\nabla z^n+\mathcal{P}_{\rm div}(V_n^{(1)})$, by adding $\nabla y^n$ to both sides and taking the divergence we get that
			\begin{align}
				-\Delta(z^n-y^n)={\rm div}(\nabla y^n-V_n^{(1)}).
			\end{align}
			Then, by standard elliptic estimates, for all $1<m<\infty$ it holds that
			\begin{align}\label{eq:loc6}
				\|\nabla z^n(t,\cdot)-\nabla y^n(t,\cdot)\|_{L^m(Q)}\lesssim \|\nabla y^n(t,\cdot)-V_n^{(1)}(t,\cdot)\|_{L^m(Q)}.
			\end{align}
			In our setting we treat the case $d=3$ and so letting $m=3p/(p+3)$, by Sobolev embedding and \eqref{eq:loc6} we have that
			\begin{align*}
				\|z^n(t,\cdot)- y^n(t,\cdot)\|_{L^p(Q)}\lesssim  \|\nabla y^n(t,\cdot)-V_n^{(1)}(t,\cdot)\|_{L^m(Q)},
			\end{align*}
			and by integrating in time,
			\begin{align*}
				\int_0^T \|z^n(t,\cdot)- y^n(t,\cdot)\|_{L^p(Q)}^p dt\lesssim \int_0^T \|\nabla y^n(t,\cdot)-V_n^{(1)}(t,\cdot)\|_{L^m(Q)}^p\to 0,
			\end{align*}
			as $n\to\infty$. The last convergence comes from \eqref{eq:loc7} and concludes the proof of the theorem since, from \eqref{eq:loc2}, $y^n\to y(t_0,\cdot)$ in $L^p(Q_T)$.

		\end{proof}
	\end{Theorem}

	\section{G\aa rding inequality}
	In this section, and more precisely in Theorem \ref{theorem:3Garding}, we prove the G\aa rding inequality \eqref{eq:mmgarding} which plays a crucial role in the proof of our weak-strong uniqueness result, Theorem \ref{Uniqueness_quasiconvex}. We collect all continuous functions $\bar{F}:\mathbb{R}^d\to\mathbb{R}^{d\times d}$ and $\bar{\eta}:\mathbb{R}^d\to\mathbb{R}$ in the ball of $L^\infty(Q)$ of radius $K$, with uniform modulus of continuity $\omega$, in the set
	\begin{align*}
		\mathcal{U}_K:=\left\{(\bar{F},\bar{\eta})\in C_K(Q):|\bar{F}(x)-\bar{F}(y)|+|\bar{\eta}(x)-\bar{\eta}(y)|\leq\omega(|x-y|),\forall x,y\in\overline{Q}\right\},
	\end{align*}
	where $C_K(Q):=\left\{(\bar{F},\bar{\eta})\in C(Q;\mathbb{R}^{d\times d})\times C(Q;\mathbb{R}):\|\bar{F}\|_{L^\infty(Q)}+\|\bar{\eta}\|_{L^\infty(Q)}\leq K\right\}$. 
	\par
	\vspace{0.2cm}
	Next, for $(z_1,z_2)\in\mathbb{R}^{d\times d}\times\mathbb{R}$ and $e\in C^3(\mathbb{R}^{d\times d}\times\mathbb{R})$ which satisfies the growth conditions $(H_2)$ and $(H_3)$,  we  define the function
	\[
	\tilde{e}(z_1,z_2):=e(z_1,z_2)-C_1|V_p(z_1)|^2-C_2|V_q(z_2)|^2,
	\]
	which is not hard to check that satisfies the same growth and coercivity conditions with $e$ up to smaller positive constants. The corresponding Hessians of $e$ and $\tilde{e}$ are denoted by $L$ and $\tilde{L}$ respectively, i.e.
	\begin{align*}
		L(\lambda_1,\lambda_2)[(\xi_1,\xi_2),(\xi_1,\xi_2)]:=e_{FF}(\lambda_1,\lambda_2)\xi_1\xi_1+2e_{\eta F}(\lambda_1,\lambda_2)\xi_1\xi_2+e_{\eta\eta}(\lambda_1,\lambda_2)\xi_2\xi_2,
	\end{align*}
	for all $(\xi_1,\xi_2)\in\mathbb{R}^{d\times d}\times\mathbb{R}$ and $(\lambda_1,\lambda_2)\in \overline{B(0,K)}:=\{\lambda\in\mathbb{R}^{d\times d}\times\mathbb{R}:\,|\lambda|\leq K\}$.  In the sequel, without loss of generality, we assume that $p\geq q$. In the opposite case, i.e. if $p<q$, the results below can be proved following the same strategy but with the respective adjustments in the proofs.

	\vspace{0.3cm}
	\par
	We next prove a series of results which lead to the proof of Theorem \ref{thm:main}. Lemma \ref{growths} provides some properties of the relative function $e(\cdot|\cdot)$ and its proof can be found in the Appendix.  For brevity, henceforth, constants shown to depend on $K$, e.g. $C(K)$, may also depend on the modulus of continuity $\omega$ but the latter dependence is omitted from the notation.
	
	\begin{Lemma}\label{growths}
		Let $f$ satisfy $(H_1),\,(H_2)$ and $(H_3)$. Then the following hold: 
		\begin{itemize}
			\item[(a)] There exists $C=C(f,K)$ such that for all $(\lambda_1,\lambda_2)\in \overline{B(0,K)}$
			\begin{align*}
				&|f(\lambda_1+\xi_1,\lambda_2+\xi_2|\lambda_1,\lambda_2) - f(\lambda_1+z_1,\lambda_2+z_2|\lambda_1,\lambda_2)| \\
				&\leq C (|\xi_1|+|\xi_2|+|z_1| + |z_2| +|\xi_1|^{p-1}+ |z_1|^{p-1}+|\xi_2|^{q\frac{p-1}{p}})|\xi_1-z_1|\\
				&+C(|\xi_1|+|\xi_2|+|z_1| + |z_2| +|\xi_2|^{q-1}+ |z_2|^{q-1}+|z_1|^{p\frac{q-1}{q}})|\xi_2-z_2|.
			\end{align*}
			Additionally,
			\[
			|f(\lambda_1+\xi_1,\lambda_2+\xi_2|\lambda_1,\lambda_2)|\leq C\big(|V_p(\xi_1)|^2+|V_q(\xi_2)|^2\big).
			\]
			\item[(b)] For every $\delta>0$ there exists $R = R(\delta,f,K)>0$ such that for all $(\lambda_1,\lambda_2),$ $(\mu_1,\mu_2)\in \overline{B(0,K)}$ with $|(\lambda_1,\lambda_2)-(\mu_1,\mu_2)|<R$, it holds that
			\begin{align*}
				|f(\lambda_1+\xi_1,\lambda_2+\xi_2|\lambda_1,\lambda_2) - f(\mu_1+\xi_1,\mu_2+\xi_2|\mu_1,\mu_2)|\\
				\leq \delta \big(|V_p(\xi_1)|^2+|V_q(\xi_2)|^2\big).
			\end{align*}
			\item[(c)] There exist constants $d_1=d_1(f,K)$, $d_2=d_2(f,K)$ such that for all $(\lambda_1,\lambda_2)\in \overline{B(0,K)}$
			\[
			f(\lambda_1+\xi_1,\lambda_2+\xi_2|\lambda_1,\lambda_2) \geq d_1 \big(|\xi_1|^p+|\xi_2|^q\big) - d_2 \big(|\xi_1|^2+|\xi_2|^2\big).
			\]
		\end{itemize}
	\end{Lemma}

	Next, we prove two important properties of the function $\tilde{e}$. In the first lemma below, Lemma \ref{teqc}, we show that $\tilde{e}$ retains the key quasiconvexity property of $e$ in $\overline{B(0,K)}$, and as a consequence of this result we next prove in Lemma \ref{hsnfx} that the Hessian $\tilde{L}$ is positive for fixed $x_0\in Q$. 
	\begin{Lemma}\label{teqc}
		The function $\tilde{e}$ is strongly quasiconvex at all $(\lambda_1,\lambda_2)\in \overline{B(0,K)}$ with constant $c_0/2$, i.e. for any $Q^\prime \subseteq Q$ and all $|\lambda_1|+|\lambda_2|\leq K$
		\begin{equation*}
			\int_{Q^\prime} \tilde{e}(\lambda_1+\nabla\phi,\lambda_2+\psi)-\tilde{e}(\lambda_1,\lambda_2)-\tilde{e}_\eta(\lambda_1,\lambda_2)\psi\geq c_0\int_{Q^\prime}|V_p(\nabla\phi)|^2+|V_q(\psi)|^2,
		\end{equation*}
		holds for all $\phi\in W^{1,p}_0(Q^\prime)$ and $\psi\in L^q(Q^\prime)$.
	\end{Lemma}
	\begin{proof}
		Let $Q'\subseteq Q$, $\phi\in W^{1,p}_0(Q')$, $\psi\in L^q(Q')$ and $(\lambda_1,\lambda_2)\in \overline{B(0,K)}$. Then 
		\begin{align*}
			&\int_{Q'}\tilde e(\lambda_1+\nabla\phi,\lambda_2+\psi)-\tilde e(\lambda_1,\lambda_2)-\tilde e_\eta(\lambda_1,\lambda_2)\psi\\
			&=\int_{Q'}e(\lambda_1+\nabla\phi,\lambda_2+\psi)-e(\lambda_1,\lambda_2)-e_\eta(\lambda_1,\lambda_2)\psi\\
			&-C_1\int_{Q'}|\lambda_1+\nabla\phi|^p-|z_1|^p+|\lambda_1+\nabla\phi|^2-|\lambda_1|^2\\
			&-C_2\int_{Q'}|\lambda_2+\psi|^q-|\lambda_2|^q+|\lambda_2+\psi|^2-|\lambda_2|^2-q\lambda_2|\lambda_2|^{q-2}\psi-2\lambda_2\psi
		\end{align*}
		\begin{align*}
			&=\int_{Q'}e(\lambda_1+\nabla\phi,\lambda_2+\psi)-e(\lambda_1,\lambda_2)-e_\eta(\lambda_1,\lambda_2)\psi\\
			&-C_1\int_{Q'}|\lambda_1+\nabla\phi|^p-|\lambda_1|^p+|\nabla\phi|^2\\
			&-C_2\int_{Q'}|\lambda_2+\psi|^q-|\lambda_2|^q-q\lambda_2|\lambda_2|^{q-2}\psi+|\psi|^2=:I_1+I_2+I_3.
		\end{align*}
		By the quasiconvexity of $e$ we infer that
		\begin{equation*}
			I_1\geq c_0\int_{Q'}|\nabla\phi|^p+|\nabla\phi|^2+|\psi|^q+|\psi|^2,
		\end{equation*}
		and for $f(\cdot)=|V_i(\cdot)|^2$ with $i=p$ and $i=q$ respectively in Lemma \ref{growths} (a), we deduce that
		\begin{align*}
			I_2&\geq - C_1 C\int_{Q'}|\nabla\phi|^p+|\nabla\phi|^2-C_1\int_{Q'}|\nabla\phi|^2,\\
			I_3&\geq -C_2 C\int_{Q'}|\psi|^p+|\psi|^2-C_2\int_{Q'}|\psi|^2.
		\end{align*}
		So, we may choose $C_1\leq c_0/(2C)$ and $C_2\leq c_0/(2C)$ to conclude the proof.
	\end{proof}
	As a consequence of the above lemma, in the result below, we deduce the positivity of the Hessian $\tilde L$ for fixed $x_0\in Q$.
	\begin{Lemma}\label{hsnfx}
		Let $Q'\subseteq Q$ and $x_0\in Q'$. Then for all $\phi\in W^{1,p}_0(Q')$ and $\psi\in L^q(Q')$ it holds that
		\begin{equation*}
			\int_{Q'}\tilde L(\bar{F}_0,\bai_0)\big[(\nabla\phi(x),\psi(x)),(\nabla\phi(x),\psi(x))\big]dx\geq c_0\int_{Q'}|\nabla\phi(x)|^2+|\psi(x)|^2dx,
		\end{equation*}
		where $\bar{F}_0=\baf(x_0)$ and $\bai_0=\bai(x_0)$.
	\end{Lemma}
	\begin{proof}
		The quasiconvexity of $\tilde{e}$, Lemma \ref{teqc}, says that  $I(\phi,\psi)\geq I(0,0)$ for all $\phi\in W^{1,p}_0(Q')$ and $\psi\in L^q(Q')$, where 
		\begin{align*}
			I(\phi,\psi)&:=\int_{Q'}\tilde e(\baf_0+\nabla\phi,\bai_0+\psi)-\tilde e(\baf_0,\bai_0)-\tilde e_\eta(\baf_0,\bai_0)\psi\\
			&-\frac{c_0}{2}\int_{Q'}|V_p(\nabla\phi)|^2+|V_q(\psi)|^2,
		\end{align*}
		and so we infer that $\frac{d^2}{d\eps^2}I(\eps\phi,\eps\psi)\Big|_{\eps=0}\geq 0$. However,
		\begin{align*}
			\frac{d}{d\eps}I(\eps\phi,\eps\psi)&:=\int_{Q'}\tilde e_F(\baf_0+\eps\nabla\phi,\bai_0+\eps\psi)\nabla\phi+\tilde e_\eta(\baf_0+\eps\nabla\phi,\bai_0+\eps\psi)\psi\\
			&-\tilde e_\eta(\baf_0,\bai_0)\psi-c_0\eps|\nabla\phi|^2-\frac{c_0}{2}p\eps^{p-1}|\nabla\phi|^p-c_0\eps |\psi|^2-\frac{c_0}{2}q\eps^{q-1}|\psi|^q,
		\end{align*}
		and so
		\begin{align*}
			\frac{d^2}{d\eps^2}I(\eps\phi,\eps\psi)=&\int_{Q'}\tilde e_{FF}(\baf_0+\eps\nabla\phi,\bai_0+\eps\psi)\nabla\phi:\nabla\phi+\tilde e_{\eta\eta}(\baf_0+\eps\nabla\phi,\bai_0+\eps\psi)\psi\cdot\psi\\
			&+2\tilde e_{F\eta}(\baf_0+\eps\nabla\phi,\bai_0+\eps\psi)\nabla\phi\cdot\psi\\
			&-c_0|\nabla\phi|^2-\frac{c_0}{2}p(p-1)\eps^{p-2}|\nabla\phi|^p-c_0|\psi|^2-\frac{c_0}{2}q(q-1)\eps^{q-2}|\psi|^q.
		\end{align*}
		We conclude that
		\begin{align*}
			0\leq \frac{d^2}{d\eps^2}I(\eps\phi,\eps\psi)\Big|_{\eps=0}= \int_{Q'}\tilde L(\bar{F}_0,\bai_0)\big[(\nabla\phi,\psi),(\nabla\phi,\psi)\big]-c_0|\nabla\phi|^2-c_0|\psi|^2.
		\end{align*}
	\end{proof}
	We are now ready to prove a G\aa rding-type inequality for the delocalised version of the Hessian $\tilde L$, which is crucial for the contradiction argument of the proof of Theorem \ref{thm:main}.
	\begin{Proposition}\label{prop:Daf}
		For every $\delta>0$ there exist constants $c>0$ and $C_{pen}:=C(\delta)>0$ such that
		\begin{equation*}
			\int_{Q'}\tilde L(\bar{F}(x),\bai(x))\big[(\nabla\phi,\psi),(\nabla\phi,\psi)\big]\geq  c(1-\delta)^2\int_{Q'}|\nabla\phi|^2+|\psi|^2-C_{pen}\int_{Q'}|\phi|^2
		\end{equation*}
		for all $\phi\in W^{1,p}(Q)$, $\psi\in L^q(Q)$ and $(\baf,\bai)\in \mathcal{U}_K$.
	\end{Proposition}
	\begin{proof}
		Fix $\delta>0$ and pick a finite cover $\{Q_i\}_i,\,Q_i:=Q(x_i,r_i)\subseteq Q$ such that
		\begin{align*}
			&|\tilde e_{FF}(\baf(x),\bai(x))-\tilde e_{FF}(\baf(x_i),\bai(x_i))|+|\tilde e_{\eta\eta}(\baf(x),\bai(x))-\tilde e_{\eta\eta}(\baf(x_i),\bai(x_i))|\\
			&+2|\tilde e_{F\eta}(\baf(x),\bai(x))-\tilde e_{F\eta}(\baf(x_i),\bai(x_i))|\leq \frac{1}{2}c\delta(1-\delta).
		\end{align*}
		Note that since $(\baf(x),\bai(x))\in \mathcal{U}_K$ (bounded with uniform modulus of continuity) and $\tilde e\in C^2$ the cover can be chosen uniformly. 
		\par
		Now choose a partition of unity $(\rho_i)_i,\,\rho_i\in C^\infty_c(Q_i)$ and $\sum_i\rho_i^2=1$. Given $\phi\in W^{1,p}(Q)$ and $\psi\in L^q(Q)$, we infer that
		\begin{align*}
			&\sum_i\int_{Q_i}\tilde L(\baf,\bai)[(\rho_i\nabla\phi,\rho_i\psi),(\rho_i\nabla\phi,\rho_i\psi)]-\tilde L(\baf_i,\bai_i)[(\rho_i\nabla\phi,\rho_i\psi),(\rho_i\nabla\phi,\rho_i\psi)]\\
			&\geq -\frac{c}{2}\delta(1-\delta)\int_Q\big(|\nabla\phi|^2+|\nabla\phi||\psi|+|\psi|^2\big)\geq -c\delta(1-\delta)\int_Q\big(|\nabla\phi|^2+|\psi|^2\big),
		\end{align*}
		and so,
		\begin{align*}
			&\int_Q\tilde L(\baf,\bai)[(\nabla\phi,\psi),(\nabla\phi,\psi)]=\\
			&\sum_i\int_{Q_i}\tilde L(\baf,\bai)[(\rho_i\nabla\phi,\rho_i\psi),(\rho_i\nabla\phi,\rho_i\psi)]-\tilde L(\baf_i,\bai_i)[(\rho_i\nabla\phi,\rho_i\psi),(\rho_i\nabla\phi,\rho_i\psi)]\\
			&+\sum_i\int_{Q_i}\tilde L(\baf_i,\bai_i)[(\rho_i\nabla\phi,\rho_i\psi),(\rho_i\nabla\phi,\rho_i\psi)]=:\sum_iI_i+\sum_iII_i.
		\end{align*}
		Since we already bounded the term $\sum_iI_i$ from below, it remains to prove a similar bound for the second term. Since $\nabla(\rho_i\phi)=\rho_i\nabla\phi+\phi\otimes\nabla\rho_i$ we infer that
		\begin{align*}
			II_i=&\int_{Q_i}\tilde e_{FF}(\baf_i,\bai_i)\nabla(\rho_i\phi):\nabla(\rho_i\phi)+\tilde e_{\eta\eta}(\baf_i,\bai_i)(\rho_i\psi):(\rho_i\psi)\\
			&+2\tilde e_{F\eta}(\baf_i,\bai_i)\nabla(\rho_i\phi):(\rho_i\psi)+\int_{Q_i}\tilde e_{FF}(\baf_i,\bai_i)(\phi\nabla\rho_i):(\phi\otimes\nabla\rho_i)\\
			&-\int_{Q_i}\tilde e_{FF}(\baf_i,\bai_i)\nabla(\rho_i\phi):(\phi\otimes\nabla\rho_i)-\int_{Q_i}\tilde e_{F\eta}(\baf_i,\bai_i)(\phi\otimes\nabla\rho_i):(\rho_i\psi)=:\sum_{j=1}^4T_i^j.
		\end{align*}
		In particular, for $T
		^1_i$ we apply Lemma \ref{hsnfx} testing with $\rho_i\phi\in W^{1,p}_0(Q_i)$ and $\rho_i\psi\in L^q(Q_i)$ to infer that
		\begin{equation*}
			T_i^1\geq c_0\int_{Q_i}\big(|\nabla(\rho_i\phi)|^2+|\rho_i\psi|^2\big).
		\end{equation*}
		For the remaining three terms, by Young's inequality, we find that
		\begin{align*}
			T_i^2&\geq -c\|\nabla\rho_i\|_\infty\int_{Q_i}|\phi|^2,\\
			T^3_i&\geq -c(1-\delta)^2\int_{Q_i}|\nabla(\rho_i\phi)|^2-2C_1(\delta)\int_{Q_i}|\phi|^2,\\
			T_i^4&\geq -c\frac{1+\delta^2}{2}\int_{Q_i}|\rho_i\psi|^2-2C_2(\delta)\int_{Q_i}|\phi|^2.
		\end{align*}
		Finally combining all the above and since
		\begin{align*}
			|\nabla(\rho_i\phi)|^2=|\rho_i\nabla\phi+\phi\otimes\nabla\rho_i|^2&\leq \frac{1}{2}|\rho_i\nabla\phi|^2+\frac{1}{2}|\phi\otimes\nabla\rho_i|^2,\\
			|\nabla(\rho_i\phi)|^2=|\rho_i\nabla\phi+\phi\otimes\nabla\rho_i|^2&\geq(1-\delta)\rho_i^2|\nabla\phi|^2-C(\delta)|\phi\otimes\nabla\rho_i|^2,
		\end{align*}
		we conclude that
		\begin{align*}
			\int_Q\tilde L(\baf,\bai)[(\nabla\phi,\psi),(\nabla\phi,\psi)]\gtrsim (1-\delta)^2\int_{Q'}|\nabla\phi|^2+|\psi|^2-C(\delta)\int_{Q'}|\phi|^2.
		\end{align*}
	\end{proof}
	The following lemma shows that the function $\tilde{e}(\cdot|\cdot)$ has convex-like behaviour on the wave cone when it is integrated over cubes with sufficiently small radius and plays a crucial role in the proof of Proposition \ref{prop: lgrd}.
	\begin{Lemma}\label{zhang}
		There exists $R=R(c_0,K)>0$ such that for all $x_0\in Q$ the inequality 
		\begin{align*}
			\int_{Q(x_0,R)}\tilde e(\baf(x)+\nabla\phi,\bai(x)+\psi|\baf(x),\bai(x))\geq \frac{c_0}{4}\int_{Q(x_0,R)}|V_p(\nabla\phi)|^2+|V_q(\psi)|^2
		\end{align*}
		holds for all $\phi\in W^{1,p}_0(Q(x_0,r))$ and $\psi\in L^q(Q(x_0,r))$ with $r\leq R$.
	\end{Lemma}
	\begin{proof}
		Observe that by Lemma \ref{growths} (b), letting $\delta=c_0/4$ we find $R=R(c_0,\tilde e,K)$ such that for all $(\baf,\bai)\in \mathcal{U}_K$ and whenever $|x-x_0|<R$
		\begin{equation*}
			\big|\tilde e(\baf+\nabla\phi,\bai+\psi|\baf,\bai)-\tilde e(\baf_0+\nabla\phi,\bai_0+\psi|\baf_0,\bai_0)\big|\leq \frac{c_0}{4}\big(|V_p(\nabla\phi)|^2+|V_q(\psi)|^2\big).
		\end{equation*}
		Note here that $\tilde e$ satisfies the growths of Lemma \ref{growths}. So for $\phi\in W^{1,p}_0(Q(x_0,r))$ and $\psi\in L^q(Q(x_0,r))$ with $r\leq R$ we infer that 
		\begin{align*}
			\int_{Q(x_0,R)}\tilde e(\baf+\nabla\phi,\bai+\psi|\baf,\bai)\geq \int_{Q(x_0,R)}&e(\baf_0+\nabla\phi,\bai_0+\psi|\baf_0,\bai_0)\\
			&-\frac{c_0}{4}\big(|V_p(\nabla\phi)|^2+|V_q(\psi)|^2\big)
		\end{align*}
		\begin{align*}
			=\int_{Q(x_0,R)}e(\baf_0+\nabla\phi,\bai_0+\psi)-e(\baf_0,\bai_0)-&e_F(\baf_0,\bai_0)\nabla\phi-e_\eta(\baf_0,\bai_0)\psi\\
			-\frac{c_0}{4}\big(|V_p(\nabla\phi)|^2+|V_q(\psi)|^2\big)
			&\geq \frac{c_0}{4}\big(|V_p(\nabla\phi)|^2+|V_q(\psi)|^2\big),
		\end{align*}
		where in the last inequality we used Lemma \ref{teqc} and that $\int_{Q(x_0,R)}\nabla\phi=0$.
	\end{proof}
	We next prove a central result which can be seen as a limiting version of a G\aa rding inequality and replaces the quasiconvexity condition in the proof of Theorem \ref{thm:main}.
	\begin{Proposition}\label{prop: lgrd}
		Let $(\baf_k,\bai_k)_k\subseteq \mathcal{U}_K$, $(\phi_k)_k\subseteq W^{1,p}(Q)$, $(\psi_k)_k\subseteq L^q(Q)$ and $(a_k)_k\subseteq\mathbb{R}$ such that 
		\begin{itemize}
			\item $\,a_k^{-1}V_p(\phi_k) \rightarrow 0$ strongly in $L^2(Q)$,
			\item $\left(a_k^{-1}V_p(\nabla\phi_k)\right)_k$ is bounded in $L^2(Q)$,
			\item $\left(a_k^{-1}V_q(\psi_k)\right)_k$ is bounded in $L^2(Q)$.
		\end{itemize}
		Then,
		\begin{align*}
			\liminf_k\frac{c_0}{4}a_k^{-2}\int_Q|V_p(\nabla\phi_k)|^2+&|V_q(\psi_k)|^2\\
			&\leq \liminf_k a_k^{-2}\int_Q\tilde e(\baf_k+\nabla\phi_k,\bai_k+\psi_k|\baf_k,\bai_k).
		\end{align*}
	\end{Proposition}
	\begin{proof}
		Since $\left(a_k^{-2}|V_p(\nabla\phi_k)|^2+a_k^{-2}|V_q(\psi_k)|^2\right)_k$ is bounded in $L^1(Q)$ we may assume that (up to a subsequence)
		\begin{equation*}
			a_k^{-2}|V_p(\nabla\phi_k)|^2+a_k^{-2}|V_q(\psi_k)|^2 \mathcal{L}^d\mres Q \overset{\ast}\rightharpoonup \mu,\quad\mbox{in }\mathcal{M}(Q) = \left(C(Q)\right)^\ast.
		\end{equation*}
		Since $\mu$ is a positive measure, there can be at most a countable number of hyperplanes parallel to the coordinate axes which admit posotive $\mu$-measure. Hence, we can extract a finite cover of $Q$ by non overlapping cubes $Q_{r_j}:=Q(x_j,r_j)$ with the property that $r_j<R$, so that Lemma \ref{zhang} applies and that
		\begin{equation}\label{eq:muzero}
			\mu( \partial Q(x_j,r_j)) = 0.
		\end{equation}
		Next, consider cut-off functions $\rho_j \in C^\infty_c(Q(x_j,r_j))$ such that for $\lambda\in(0,1)$
		\[
		\mathbbm{1}_{Q(x_j,\lambda r_j)} \leq \rho_j \leq \mathbbm{1}_{Q(x_j, r_j)},\,\, \|\nabla\rho_j\|_{L^{\infty}(Q)}\leq \frac{C}{1-\lambda}.
		\]
		For simplicity, we denote $Q_{r_j}:=Q(x_j,r_j)$ and $Q_{\lambda r_j}:=Q(x_j,\lambda r_j)$. We now apply Lemma \ref{zhang} for the functions $\rho_j\phi_k\in W^{1,p}_0(Q_{ r_j})$ and $\rho_j\psi_k\in L^q(Q_{ r_j})$ to find that
		\begin{align*}
			\frac{c_0}{4}\int_{Q_{r_j}}|V_p(\nabla(\rho_j\phi_k))|^2+|V_q(\rho_j\psi_k)|^2\leq \int_{Q_{r_j}}\tilde e(\baf_k+\nabla(\rho_j\phi_k),\bai_k+\rho_j\psi_k|\baf_k,\bai_k),
		\end{align*}
		where $(\baf_k,\bai_k)\in \mathcal{U}_K$. Thus by Lemma \ref{growths} (a) and for $C=C(\tilde e,K)$, it holds that
		\begin{align*}
			&\frac{c_0}{4}\int_{Q_{\lambda r_j}}|V_p(\nabla\phi_k)|^2+|V_q(\psi_k)|^2+\frac{c_0}{4}\int_{Q_{r_j}\setminus Q_{\lambda r_j}}|V_p(\nabla(\rho_j\phi_k))|^2+|V_q(\rho_j\psi_k)|^2\\
			&\leq \int_{Q_{\lambda r_j}}\tilde e(\baf_k+\nabla\phi_k,\bai_k+\psi_k|\baf_k,\bai_k)\\
			&\hspace{4cm}+\int_{Q_{r_j}\setminus Q_{\lambda r_j}}\tilde e(\baf_k+\nabla(\rho_j\phi_k),\bai_k+\rho_j\psi_k|\baf_k,\bai_k)\\
			&\leq \int_{Q_{\lambda r_j}}\tilde e(\baf_k+\nabla\phi_k,\bai_k+\psi_k|\baf_k,\bai_k)\\
			&\hspace{4cm}+C\int_{Q_{r_j}\setminus Q_{\lambda r_j}}|V_p(\nabla(\rho_j\phi_k))|^2+|V_q(\rho_j\psi_k)|^2\\
			&\leq \int_{Q_{\lambda r_j}}\tilde e(\baf_k+\nabla\phi_k,\bai_k+\psi_k|\baf_k,\bai_k)\\
			&\hspace{3.7cm}+C\int_{Q_{r_j}\setminus Q_{\lambda r_j}}|V_p(\nabla\phi_k)|^2+|V_q(\psi_k)|^2+\Big|V_p\Big(\frac{\phi_k}{1-\lambda}\Big)\Big|^2,
		\end{align*}
		where in the last inequality we used the definition of the cut-offs and the fact that $\nabla(\rho_j\phi_k)=\phi_k\otimes\nabla\rho_j+\rho_j\nabla\phi_k$. We observe that the second term of the left hand side is positive and so by summing over $j$  we infer that
		\begin{align*}
			&\frac{c_0}{4}\int_Q|V_p(\nabla\phi_k)|^2+|V_q(\psi_k)|^2-\frac{c_0}{4}\sum_j\int_{Q_{r_j}\setminus Q_{\lambda r_j}}|V_p(\nabla\phi_k)|^2+|V_q(\psi_k)|^2\\
			&\leq \int_Q\tilde e(\baf_k+\nabla\phi_k,\bai_k+\psi_k|\baf_k,\bai_k)-\sum_j\int_{Q_{r_j}\setminus Q_{\lambda r_j}}\tilde e(\baf_k+\nabla\phi_k,\bai_k+\psi_k|\baf_k,\bai_k)\\
			&+C\sum_j\int_{Q_{r_j}\setminus Q_{\lambda r_j}}|V_p(\nabla\phi_k)|^2+|V_q(\psi_k)|^2+\Big|V_p\Big(\frac{\phi_k}{1-\lambda}\Big)\Big|^2.
		\end{align*}
		Using again Lemma \ref{growths} (a) we deduce that
		\begin{align*}
			&\frac{c_0}{4}\int_Q|V_p(\nabla\phi_k)|^2+|V_q(\psi_k)|^2\leq \int_Q\tilde e(\baf_k+\nabla\phi_k,\bai_k+\psi_k|\baf_k,\bai_k)\\
			&\hspace{3cm}+C\sum_j\int_{Q_{r_j}\setminus Q_{\lambda r_j}}|V_p(\nabla\phi_k)|^2+|V_q(\psi_k)|^2+\Big|V_p\Big(\frac{\phi_k}{1-\lambda}\Big)\Big|^2.
		\end{align*}
		Multiplying with $a_k^{-2}$ and taking the limit over $k$ we conclude that
		\begin{align*}
			\liminf_k \frac{c_0}{4}a_k^{-2}\int_Q|V_p(\nabla\phi_k)|^2&+|V_q(\psi_k)|^2&\\
			&\leq a_k^{-2} \liminf_k\int_Q\tilde e(\baf_k+\nabla\phi_k,\bai_k+\psi_k|\baf_k,\bai_k)\\
			&\hspace{1.5cm}+C\sum_j\mu\big(\bar Q_{r_j}\setminus Q_{\lambda r_j}\big).
		\end{align*}
		Finally, send $\lambda\to 1$ to complete the proof via \eqref{eq:muzero}.
	\end{proof}
	The G\aa rding inequality, Theorem \ref{theorem:3Garding}, follows as a consequence of Theorem \ref{thm:main} below which forms the core of this section.
	\begin{Theorem}\label{thm:main}
		Assume that $e$ satisfies $(H_1),(H_2),(H_3)$. There exists $\epsilon_0>0$ and constants $\tilde C_0=\tilde C_0(e,K)>0$, $\tilde C_1=\tilde C_1(e,K)>0$ such that for all $(\baf,\bai)\in \mathcal{U}_K$, $\phi\in W^{1,p}_0(Q)$ and $\psi\in L^q(Q)$ with $\int_Q\psi=0$ and $\|\phi\|_{L^p(Q)}<\epsilon_0$ it holds that
		\begin{align*}
			\int_Q&\Big(|V_p(\nabla\phi(x))|^2+|V_q(\psi(x))|^2\Big)dx \\
			&\leq \tilde C_0\int_Q e(\baf(x)+\nabla\phi(x),\bai(x)+\psi(x)|\baf(x),\bai(x))dx+ \tilde C_1\int_Q|V_p(\phi(x))|^2dx.
		\end{align*}
	\end{Theorem}
	\begin{proof}
		It is enough to show the existence of some $\epsilon_0>0$ such that for all $(\baf,\bai)\in \mathcal{U}_K$, $\psi\in L^q(Q)$ with zero average and $\phi\in W^{1,p}_0(Q)$ with $\|\phi\|_{L^p(Q)}<\epsilon_0$ it holds that
		\begin{align}\label{eq:cinq}
			\int_Q \tilde e(\baf+\nabla\phi,\bai+\psi|\baf,\bai)+\frac{C_{pen}}{2}|\phi|^2\geq 0.
		\end{align}
		Indeed, due to the convexity of the functions $|V_p(\cdot)|^2$ and $|V_q(\cdot)|^2$, we see that \eqref{eq:cinq} implies the required inequality. 
		\par 
		We proceed by contradiction. Suppose that \eqref{eq:cinq} fails. Then, there exist $(\baf_k,\bai_k)_k\subseteq \mathcal{U}_K$, $(\psi_k)_k\subseteq L^q(Q)$ zero-average and $(\phi_k)_k\subseteq W^{1,p}_0(Q)$ with $\|\phi_k\|_{L^p(Q)}\to 0$ and  $(\baf_k,\bai_k) \to (\baf,\bai)$ in $C^0(Q)$ such that 
		\begin{equation}\label{eq:cntr}
			\int_Q \tilde e(\baf_k+\nabla\phi_k,\bai_k+\psi_k|\baf_k,\bai_k)+\frac{C_{pen}}{2}|\phi_k|^2< 0.
		\end{equation}
		Let $\beta_k^\eta:=\Big(\|\nabla\phi_k\|^q_{L^q(Q)}+\|\psi_k\|^q_{L^q(Q)}\Big)^{1/q},\,\,\,\beta_k^F:=\Big(\|\nabla\phi_k\|^p_{L^p(Q)}+\|\psi_k\|^q_{L^q(Q)}\Big)^{1/p}$\\
		and $\al_k:=\Big(\|\nabla\phi_k\|^2_{L^2(Q)}+\|\psi_k\|^2_{L^2(Q)}\Big)^{1/2}$.\\
		\vsp
		\par
		\noindent\underline{Step 1}: We show that $\|\nabla\phi_k\|_{L^p(Q)}\to 0,\, \|\psi_k\|_{L^q(Q)}\to 0$ as $k\to\infty$ and
		\begin{equation}\label{eq:multi}
			\Lambda^F:=\sup_k\frac{({\beta_k^F})^p}{\al_k^2}<\infty,\hspace{1cm}
			\Lambda^\eta:=\sup_k\frac{({\beta_k^\eta})^q}{\al_k^2}<\infty.
		\end{equation}
		Indeed, by Lemma \ref{growths} (c) and Young's inequality we infer that
		\begin{align*}
			&\tilde e(\baf_k+\nabla\phi_k,\bai_k+\psi_k|\baf_k,\bai_k)=\tilde e(\baf_k+\nabla\phi_k,\bai_k+\psi_k)-\tilde e(\baf_k,\bai_k)\\
			&-\tilde e_F(\baf_k,\bai_k)\nabla\phi_k-\tilde e_\eta(\baf_k,\bai_k)\psi_k\geq -C(\delta)+(\frac{c}{2^p}-\delta)|\nabla\phi_k|^p+(\frac{c}{2^q}-\delta)|\psi_k|^q,
		\end{align*}
		and for $\delta>0$ small enough we conclude that 
		\begin{align*}
			0\overset{\eqref{eq:cntr}}>\int_Q \tilde e(\baf_k+\nabla\phi_k,\bai_k+\psi_k|\baf_k,\bai_k)\geq-C+c\int_Q|\nabla\phi_k|^p+|\psi_k|^q.
		\end{align*}
		The above inequality tell us that the sequences $(\nabla\phi_k)_k$ and $(\psi_k)_k$ are bounded in $L^p(Q)$ and $L^q(Q)$ respectively. We now apply Proposition \ref{prop: lgrd} for the sequences $(\baf_k,\bai_k)_k,\,(\nabla\phi_k)_k,\,(\psi_k)_k$ and $a_k=1$ to find that
		\begin{align*}
			\liminf_k\frac{c_0}{4}\int_Q|V_p(\nabla\phi_k)|^2+|V_q(\psi_k)|^2
			\leq \liminf_k\int_Q\tilde e(\baf_k+\nabla\phi_k,\bai_k+\psi_k|\baf_k,\bai_k)
			<0,
		\end{align*}
		and so, up to a subsequence, $\|\nabla\phi_k\|_{L^p(Q)}\to 0$ and $ \|\psi_k\|_{L^q(Q)}\to 0$. 
		\par 
		For the rest of Step 1 we recall Lemma \ref{growths} (c) which tell us that
		\begin{align*}
			0\overset{\eqref{eq:cntr}}\geq & d_1\int_Q\Big(|\nabla\phi_k|^p+|\psi_k|^q\Big)
			-d_2\int_Q\Big(|\nabla\phi_k|^2+|\psi_k|^2\Big),\\
			0\overset{\eqref{eq:cntr}}\geq & d_3\int_Q\Big(|\nabla\phi_k|^q+|\psi_k|^q\Big)
			-d_4\int_Q\Big(|\nabla\phi_k|^2+|\psi_k|^2\Big),
		\end{align*}
		and so by dividing both inequalities by $\al_k^2$ we conclude the proof of this step.
		\vsp
		\par
		\noindent\underline{Step 2}: Following the strategy of \cite{JC2017}, \cite{JCK2020}, \cite{GB2009} we decompose the normalised sequences
		\begin{align*}
			s_k:=\frac{\phi_k}{\al_k}\,\,\,\,\,\,\text{and}\,\,\,\,\,\,c_k:=\frac{\psi_k}{\al_k}.
		\end{align*}
		Since $\|\nabla s_k\|_{L^2(Q)}^2+\|c_k\|_{L^2(Q)}^2=1$ we find $s\in W^{1,2}_0(Q)$ and $c\in L^2(Q)$ with $\int_Q c=0$ such that $s_k\rightharpoonup s$ in $W^{1,2}(Q)$ and $c_k\rightharpoonup c$ in $L^2(Q)$. Setting 
		\[
		M_k^F:=2^{-\frac{p-2}{2p}}\frac{\al_k}{\beta_k^F}\,\,\,\,\,\,\text{and}\,\,\,\,\,\,M_k^\eta:=2^{-\frac{q-2}{2q}}\frac{\al_k}{\beta_k^\eta},
		\]
		we also infer that $\|M_k^F\nabla s_k\|_{L^p(Q)},\,\|M_k^\eta c_k\|_{L^q(Q)}\leq 1$ and $M_k^F,\,M_k^\eta\in (0,1]$. The first two bounds come directly from the definition of the sequences $s_k$ and $c_k$, while for the rest we have that (up to a further subsequence)
		\begin{align*}
			\frac{({M_k^F})^p}{2^{-\frac{p-2}{2}}}=\frac{\al_k^p}{({\beta_k^F})^p}=\frac{\Big(\|\nabla\phi_k\|^2_{L^2}+\|\psi_k\|^2_{L^2}\Big)^{p/2}}{\|\nabla\phi_k\|^p_{L^p}+\|\psi_k\|^q_{L^q}}\leq 2^{(p-2)/2}\frac{\|\nabla\phi_k\|^p_{L^2}+\|\psi_k\|^p_{L^2}}{\|\nabla\phi_k\|^p_{L^p}+\|\psi_k\|^q_{L^q}}\leq 2^{\frac{p-2}{2}} ,
		\end{align*}
		for $k$ large enough. The last inequality comes from the fact that since $p\geq q$ and $\|\psi_k\|_{L^q}\to 0$  we may find a subsequence (denoted again by $\psi_k$) such that $\|\psi_k\|_{L^2}^p\leq \|\psi_k\|_{L^2}^q$. Similarly, we see that 
		\[
		\frac{({M_k^\eta})^q}{2^{-\frac{q-2}{2}}}=\frac{\al_k^q}{({\beta_k^\eta})^q}=\frac{\Big(\|\nabla\phi_k\|^2_{L^2}+\|\psi_k\|^2_{L^2}\Big)^{q/2}}{\|\nabla\phi_k\|^q_{L^q}+\|\psi_k\|^q_{L^q}}\leq 2^{(q-2)/2}\frac{\|\nabla\phi_k\|^q_{L^2}+\|\psi_k\|^q_{L^2}}{\|\nabla\phi_k\|^q_{L^q}+\|\psi_k\|^q_{L^q}}\leq 2^{\frac{q-2}{2}}.
		\]
		According to the Decomposition Lemmas [\cite{JC2017}, Lemma 3.4], [\cite{FLb2007}, Lemma 8.13] there exist functions $g_k\in C_c^\infty(Q),\,b_k\in W^{1,2}(Q),\,G_k\in L^2(Q)$ and $B_k\in L^2(Q)$ such that (up to common subsequences)
		\begin{enumerate}
			\item[(a)] $s_k = s + g_k + b_k$;
			\item[(b)] $g_k,\,b_k\rightharpoonup 0$ in $W^{1,2}(Q)$ and $M_k^Fg_k,\,M^F_kb_k\rightharpoonup 0$ in $W^{1,p}(Q)$;
			\item[(c)] $\big(|\nabla g_k|^2\big)_k$ and $\big(|M_k^F\nabla g_k|^p\big)_k$ are equiintegrable;
			\item[(d)] $\nabla b_k\overset{m}\rightarrow 0$ and $M_k^F\nabla b_k\overset{m}\rightarrow 0$,
		\end{enumerate}
		and
		\begin{enumerate}
			\item[(a')] $c_k = c + G_k + B_k$;
			\item[(b')] $G_k,\,B_k\rightharpoonup 0$ in $L^2(Q)$ and $M_k^\eta G_k,\,M^\eta_k B_k\rightharpoonup 0$ in $L^q(Q)$;
			\item[(c')] $\big(|G_k|^2\big)_k$ and $\big(|M_k^F G_k|^q\big)_k$ are equiintegrable;
			\item[(d')] $B_k\overset{m}\rightarrow 0$ and $M_k^\eta B_k\overset{m}\rightarrow 0$.
		\end{enumerate}
		We define
		\begin{align*}
			f_k(x):=\al_k^{-2}\Big(\tilde e(\baf_k+\al_k\nabla s_k,\bai_k+\al_k c_k|\baf_k,\bai_k)
			-\tilde e(\baf_k+\al_k\nabla b_k,\bai_k+\al_k B_k|\baf_k,\bai_k)\Big)
		\end{align*}
		and then we deduce that
		\begin{align*}
			\int_Q f_k(x)+\al_k^{-2}\tilde e(\baf_k+\al_k\nabla b_k,\bai_k+\al_k B_k|\baf_k,\bai_k)+\frac{C_{pen}}{2}|s_k|^2\\
			=\al_k^{-2}\int_Q \tilde e(\baf_k+\al_k\nabla s_k,\bai_k+\al_k c_k|\baf_k,\bai_k)+\frac{C_{pen}}{2}|\al_k s_k|^2\\
			=\al_k^{-2}\int_Q \tilde e(\baf_k+\nabla \phi_k,\bai_k+ \psi_k|\baf_k,\bai_k)+\frac{C_{pen}}{2}|\phi_k|^2\overset{\eqref{eq:cinq}}<0.
		\end{align*}
		This shows that
		\begin{equation}\label{eq:step2}
			\int_Q f_k(x)+\al_k^{-2}\tilde e(\baf_k+\al_k\nabla b_k,\bai_k+\al_k B_k|\baf_k,\bai_k)+\frac{C_{pen}}{2}|s_k|^2<0.
		\end{equation}
		\par
		\noindent\underline{Step 3}: In this step we show that the contribution of the concentrating part must be nonnegative in the limit due to quasiconvexity. In particular, we prove that
		\[
		\liminf_k \al_k^{-2}\int_Q\tilde e(\baf_k+\al_k\nabla b_k,\bai_k+\al_k B_k|\baf_k,\bai_k)\geq 0.
		\]
		For this aim we apply Proposition \ref{prop: lgrd} for the sequences $(\baf_k,\bai_k)_k,\,(\al_kb_k)_k,\,(\al_kB_k)_k$ and $(\al_k)_k$ since
		\begin{align*}
			\al_k^{-2}|V_p(\al_kb_k)|^2=|b_k|^2+\al_k^{p-2}|b_k|^p&\lesssim|b_k|^2+\Lambda^F|M_k^Fb_k|^p\to 0\,\,\,\text{in}\,\,\,L^1(Q),\\
			\sup_k\int_Q\al_k^{-2}|V_p(\al_k\nabla b_k)|^2&\lesssim\sup_k\int_Q|b_k|^2+\Lambda^F\sup_k|M_k^Fb_k|^p<\infty,\\
			\sup_k\int_Q\al_k^{-2}|V_q(\al_k B_k)|^2&\lesssim\sup_k\int_Q|B_k|^2+\Lambda^\eta\sup_k|M_k^\eta B_k|^q<\infty,
		\end{align*}
		to infer that
		\begin{align*}
			0\leq\liminf_k\frac{c_0}{4}\al_k^{-2}\int_Q|V_p(\al_k&\nabla b_k)|^2+|V_q(\al_k B_k)|^2\\
			&\leq \liminf_k \al_k^{-2}\int_Q\tilde e(\baf_k+\al_k\nabla b_k,\bai_k+\al_k B_k|\baf_k,\bai_k).
		\end{align*}
		Combining this with \eqref{eq:step2} we have that
		\begin{equation}\label{eq:step3}
			\liminf_k\int_Q f_k(x)+\frac{C_{pen}}{2}|s_k|^2<0.
		\end{equation}
		\par
		\noindent\underline{Step 4}: Next, consider the $(\rm{curl,0})$-$2$-Young measure generated by the sequence $(\nabla s_k,c_k)$, say $\nu=(\nu_x)_{x\in Q}$, and recall that $(\baf_k,\bai_k)\to(\baf,\bai)$ in $C^0(Q)$. In this step we show that
		\begin{align*}
			\frac{1}{2}\int_Q\big\langle\nu_x,\tilde L(\baf(x),\bai(x))\,[\Lambda,\Lambda]\,\big\rangle dx\leq \liminf_k\int_Q f_k(x),
		\end{align*}
		where for simplicity we use the notation $\Lambda:=(\lambda_F,\lambda_\eta)$.
		\par 
		We first prove the equiintegrability of $f_k$. By Lemma \ref{growths} (a) we find that 
		\begin{align*}
			&|f_k|=\frac{|\tilde e(\baf_k+\al_k\nabla s_k,\bai_k+\al_k c_k|\baf_k,\bai_k)
				-\tilde e(\baf_k+\al_k\nabla b_k,\bai_k+\al_k B_k|\baf_k,\bai_k)|}{\al_k^2}
		\end{align*}
		\begin{align*}
			\leq\al_k^{-2}\big(|\al_k\nabla s_k|+|\al_k& c_k|+|\al_k\nabla b_k|+|\al_k B_k|\\
			&+|\al_k\nabla s_k|^{p-1}+|\al_k\nabla b_k|^{p-1}+|\al_k B_k|^{q\frac{p-1}{p}}\big)\al_k|\nabla s+\nabla g_k|
		\end{align*}
		\begin{align*}
			+\al_k^{-2}\big(|\al_k\nabla s_k|+&|\al_k c_k|+|\al_k\nabla b_k|+|\al_k B_k|\\
			&+|\al_kc_k|^{q-1}+|\al_kB_k|^{q-1}+|\al_k \nabla b_k|^{p\frac{q-1}{q}}\big)\al_k|c+G_k|
		\end{align*}
		\begin{align*}
			=&\big(|\nabla s_k|+| c_k|+|\nabla b_k|+| B_k|+\al_k^{p-2}|\nabla s_k|^{p-1}+\al_k^{p-2}|\nabla b_k|^{p-1}\big)|\nabla s+\nabla g_k| \\
			&+\big(|\nabla s_k|+| c_k|+|\nabla b_k|+| B_k|+\al_k^{q-2}|c_k|^{q-1}+\al_k^{q-2}|B_k|^{q-1}\big)|c+G_k|\\
			&+\al_k^{-2}|\al_k B_k|^{q\frac{p-1}{p}}\al_k|\nabla s+\nabla g_k|+\al_k^{-2}|\al_k \nabla b_k|^{p\frac{q-1}{q}}\al_k|c+G_k|=:I_{1} + I_2+I_3.
		\end{align*}
		Regarding $I_{k}$, $k=1,2$, for a given set $A\subset Q$, we apply Young's inequality and we integrate both sides to infer that
		\begin{align*}
			\int_AI_{k}\leq C\delta&+C(\delta)\int_A|\nabla s+\nabla g_k|^2+\al^{p-2}|\nabla s+\nabla g_k|^p\\
			&+C(\delta)\int_A|c+G_k|^2+\al^{q-2}|c+G_k|^q.
		\end{align*}
		For $I_3$ again by Young's inequality we see that
		\begin{align*}
			I_3&=\al_k^{-2}\big(\,|\al_k B_k|^{q\frac{p-1}{p}}\al_k|\nabla s+\nabla g_k|+|\al_k \nabla b_k|^{p\frac{q-1}{q}}\al_k|c+G_k|\,\big)\\
			&\leq \al_k^{-2}\big(\,\delta\al_k^p|\nabla b_k|^p+C(\delta)\al_k^q|c+G_k|^q+\delta\al_k^q|B_k|^q+C(\delta)\al_k^p|\nabla s +\nabla g_k|^p\big)\\
			&=\delta\al_k^{p-2}|\nabla b_k|^p+C(\delta)\al_k^{q-2}|c+G_k|^q+\delta\al_k^{q-2}|B_k|^q+C(\delta)\al_k^{p-2}|\nabla s +\nabla g_k|^p.
		\end{align*}
		Combing the above we conclude that
		\begin{align*}
			\int_Af_k\leq C\delta&+C(\delta)\int_A|\nabla s+\nabla g_k|^2+\al^{p-2}|\nabla s+\nabla g_k|^p\\
			&+C(\delta)\int_A|c+G_k|^2+\al^{q-2}|c+G_k|^q.
		\end{align*}
		All sequences appearing on the right hand side are equiintegrable so, choosing $\delta>0$ appropriately, we see that $f_k$ is equiintegrable. Since $\nabla s_k,\,c_k,\,\nabla b_k,\,B_k\rightharpoonup 0$ in $L^2(Q)$, for $\eps>0$ we can find $m_\eps>0$ such that 
		\begin{align}\label{eq:eqfb}
			\int_Qf_k=\int_{A_k^c\cup B_k^c}f_k+\int_{A_k\cap B_k}f_k>-\eps+\int_{A_k\cap B_k}f_k,
		\end{align}
		for all $m\geq m_\eps$, where
		\begin{align*}
			&A_k=\big\{x\in Q:\,(|\nabla s_k|^2+|c_k|^2)^{1/2}<m\big\},\\
			&B_k=\big\{x\in Q:\,(|\nabla b_k|^2+|B_k|^2)^{1/2}<m\big\}.
		\end{align*}
		Now by choosing $m_\eps$ larger if necessary, we assume that
		\begin{align}
			\Big|\int_Q\big\langle\nu_x,\tilde L(\baf,\bai)\,[\,\Lambda,\Lambda]\mathbbm{1}_{\mathbb{R}^{d\times d}\times\mathbb{R}\setminus B(0,m)}(\Lambda)\big\rangle\Big|<\eps,  \,\,\,\,\,\text{for all}\,\,\,m\geq m_\eps.
		\end{align}
		Indeed, Young's inequality and dominated convergence give us that
		\begin{align*}
			\Big|\int_Q\big\langle\nu_x,\tilde L(\baf,\bai)\,[\,\Lambda,\Lambda &]\mathbbm{1}_{\mathbb{R}^{d\times d}\times\mathbb{R}\setminus B(0,m)}(\Lambda)\big\rangle\Big|\leq C\int_Q\big|\langle\nu_x,|\Lambda|^2\mathbbm{1}_{\mathbb{R}^{d\times d}\times\mathbb{R}\setminus B(0,m)}(\Lambda)\,\rangle\big|\\
			&=C\int_Q\big|\langle\nu_x,|\Lambda|^2\,\rangle-\langle\nu_x,|\Lambda|^2\mathbbm{1}_{B(0,m)}(\Lambda)\,\rangle\big|\longrightarrow 0
		\end{align*}
		and so,
		\begin{align*}
			&\int_Q\big\langle\nu_x,\tilde L(\baf,\bai)\,\big[\Lambda,\Lambda\big]\big\rangle=\int_Q\big\langle\nu_x,\tilde L(\baf,\bai)\,\big[\Lambda,\Lambda\big]\mathbbm{1}_{\mathbb{R}^{d\times d}\times\mathbb{R}\setminus B(0,m)}(\Lambda)\big\rangle
		\end{align*}
		\begin{align}\label{eq:htb}
			&+\int_Q\big\langle\nu_x,\tilde L(\baf,\bai)\,\big[\Lambda,\Lambda\big]\mathbbm{1}_{B(0,m)}(\Lambda)\big\rangle\leq \int_Q\big\langle\nu_x,\tilde L(\baf,\bai)\,\big[\Lambda,\Lambda\big]\mathbbm{1}_{B(0,m)}(\Lambda)\big\rangle+\eps,
		\end{align}
		for all $m\geq m_\eps$. However $\mathbbm{1}_{B(0,m)}$ is lower semicontinuous and hence, for all $x\in Q$ the function
		\[
		\Lambda\mapsto \tilde L(\baf,\bai)[\,\Lambda,\Lambda\,]\mathbbm{1}_{B(0,m)}(\Lambda)
		\]
		is also lower semicontinuous. Then, since $(\nabla s_k,c_k)$ generates $(\nu_x)_x$ we infer that
		\begin{align*}
			\int_Q\big\langle\nu_x,\tilde L(\baf,\bai)\,\big[\Lambda,\Lambda\big]\mathbbm{1}_{B(0,m)}(\Lambda)\big\rangle&\leq \liminf_k\int_{A_k}\tilde L(\baf,\bai)[\,(\nabla s_k,c_k),(\nabla s_k,c_k)\,]\\
			&=\liminf_k\int_{A_k}\tilde L(\baf_k,\bai_k)[\,(\nabla s_k,c_k),(\nabla s_k,c_k)\,]
		\end{align*}
		where the last equality follows by the strong convergence $(\baf_k,\bai_k)\to(\baf,\bai)$ in $C^0(Q)$. 
		Going back to \eqref{eq:htb} we conclude that
		\begin{align}\label{eq:htb2}
			\int_Q\big\langle\nu_x,\tilde L(\baf,\bai)\,\big[\Lambda,\Lambda\big]\big\rangle\leq \liminf_k\int_{A_k}\tilde L(\baf_k,\bai_k)[\,(\nabla s_k,c_k),(\nabla s_k,c_k)\,]+\eps,
		\end{align}
		for all $m\geq m_\eps$. We now claim that
		\begin{equation}\label{eq:claimstep4}
			\frac{1}{2}\liminf_k\int_{A_k}\tilde L(\baf_k,\bai_k)[\,(\nabla s_k,c_k),(\nabla s_k,c_k)\,]=\liminf_k\int_{A_k\cap B_k}f_k,
		\end{equation}
		for all $m\geq m_\eps$. Indeed, we observe that
		\begin{align*}
			f_k=\int_0^1(1-t&)\Big(\,\tilde L(\baf_k+t\al_k\nabla s_k,\bai_k+t\al_k c_k)[\,(\nabla s_k,c_k),(\nabla s_k,c_k)\,]\\
			&-\tilde L(\baf_k+t\al_k\nabla b_k,\bai_k+t\al_k B_k)[\,(\nabla b_k,B_k),(\nabla b_k,B_k)\,]\,\Big)dt.
		\end{align*}
		Then, since $\int_0^1(1-t)dt=1/2$, we infer that
		\begin{align*}
			\mathbbm{1}_{A_k\cap B_k} &f_k = \mathbbm{1}_{A_k\cap B_k}\int_0^1(1-t) \Big(\,\tilde L(\baf_k+t\al_k\nabla s_k,\bai_k+t\al_k c_k)[\,(\nabla s_k,c_k),(\nabla s_k,c_k)\,]\\
			&\hspace{5.3cm}-\tilde L(\baf_k,\bai_k)[\,(\nabla s_k,c_k),(\nabla s_k,c_k)\,]\,\Big)dt\\
			& + \mathbbm{1}_{A_k} \frac12 \tilde L(\baf_k,\bai_k)[\,(\nabla s_k,c_k),(\nabla s_k,c_k)\,]\\
			&- \mathbbm{1}_{A_k} \frac12 \tilde L(\baf_k,\bai_k)[\,(\nabla s_k,c_k),(\nabla s_k,c_k)\,] \left(1 - \mathbbm{1}_{B_k}\right)\\
			& - \mathbbm{1}_{A_k\cap B_k} \int_0^1(1-t)\tilde L(\baf_k+t\al_k\nabla b_k,\bai_k+t\al_k B_k)[\,(\nabla b_k,B_k),(\nabla b_k,B_k)\,] \,dt\\
			& =: I_1^k + I_2^k + I_3^k + I_4^k,
		\end{align*}
		and so it is enough to prove that 
		\[
		\lim_{k\to\infty} \int_Q I_1^k = \lim_{k\to\infty} \int_Q I_3^k = \lim_{k\to\infty} \int_Q I_4^k = 0.
		\]
		Recall that $\alpha_k\to 0$ and $(\baf_k,\bai_k) \to (\baf,\bai)$ in $C^0(Q)$. Thus, for $I_1^k$ and since we are in the set $A_k$, we find that
		\begin{align*}
			&\Big|\tilde L(\baf_k+t\al_k\nabla s_k,\bai_k+t\al_k c_k)[\,(\nabla s_k,c_k),(\nabla s_k,c_k)\,]\\
			&-\tilde L(\baf_k,\bai_k)[\,(\nabla s_k,c_k),(\nabla s_k,c_k)\,]\Big|\leq C(K)\al_k^3m^3\to 0,\,\,\,k\to \infty,
		\end{align*}
		and thus, $\int_Q I_1^k \to 0$ by dominated convergence. As for $I_3^k$, again since $\tilde L$ is continuous and $\|(\baf,\bai)\|_{L^\infty(Q)}\leq K$, we get that
		\begin{align*}
			|I_3^k| \leq C(W,K) m^2 \left(1 - \mathbbm{1}_{B_k}\right) = C(W,K) m^2  \mathbbm{1}_{B_k^c}.
		\end{align*}
		Hence, $\int_Q I_3^k \to 0$ as $\nabla b_k \to 0$ and $B_k\to 0$ in measure. We note here that
		\begin{align*}
			\{x\in Q: |\nabla b_k(x)|^2&+|B_k(x)|^2\geq \eps^2\}\\
			&\subseteq\{x\in Q: |\nabla b_k(x)|^2\geq \eps^2/2\}\cup\{x\in Q: |B_k(x)|^2\geq \eps^2/2\}.
		\end{align*}
		Lastly, for $I_4^k$, as we are in $B_k$ and $t\in(0,1)$, we get that
		\[
		(\baf_k,\bai_k) + t\alpha_k (\nabla b_k,\,B_k) \rightarrow (\baf,\bai),\,\,\,k\to\infty,
		\]
		uniformly and thus
		\begin{align*}
			|I_4^k| \leq C(W,K)(|\nabla b_k|^2+|B_k|^2)\leq  C(W,K)m(|\nabla b_k|+|B_k|) \to 0
		\end{align*}
		in measure. In particular, restricting to $B_k$, $\int_Q I_4^k \to 0$ by dominated convergence. Finally combining \eqref{eq:htb2}, \eqref{eq:claimstep4} and \eqref{eq:eqfb} we infer that
		\begin{align*}
			\frac{1}{2} \int_Q\big\langle\nu_x,\tilde L(\baf,\bai)\,\big[\Lambda,\Lambda\big]\big\rangle&\leq \liminf_k\int_{A_k\cap B_k}f_k+\eps/2\\
			&< \liminf_k\int_Qf_k+3\eps/2,
		\end{align*}
		for all $m\geq m_\eps$. Since $\eps$ is arbitrary and the dependence on $m_\eps$ has been removed, we take the limit $\eps\to 0$ to deduce that
		\begin{equation*}
			\frac{1}{2} \int_Q\big\langle\nu_x,\tilde L(\baf,\bai)\,\big[\Lambda,\Lambda\big]\big\rangle\leq \liminf_k\int_Qf_k.
		\end{equation*}
		Combining the above inequality with \eqref{eq:step3} we conclude that
		\begin{equation}\label{eq:step4}
			\frac{1}{2} \int_Q\big\langle\nu_x,\tilde L(\baf,\bai)\,\big[\Lambda,\Lambda\big]\big\rangle +C_{pen}|s_k|^2<0.
		\end{equation}
		\vsp
		\par
		\noindent\underline{Step 5}: In this step we show how \eqref{eq:step4} leads to a contradiction. By Lemma \ref{hsnfx} the function
		\[
		h(x,\xi_1,\xi_2):=\tilde L(\baf(x),\bai(x))[\,(\xi_1,\xi_2),(\xi_1,\xi_2)\,]
		\]
		is strongly quasiconvex for each $x\in Q$. We note here that our definition of quasiconvexity is equivalent with the definition of $\mathcal{A}$-quasiconvexity for $\mathcal{A}=(\rm{curl},0)$, see Remark \ref{remark:Aqc}. Hence, since $(\nabla s_k,c_k)$ generates the $(\rm{curl},0)$-2-Young measure $(\nu_x)_x$ and $h(x,\xi_1,\xi_2)$ grows quadratically in $(\xi_1,\xi_2)$, Jensen's inequality for $\mathcal{A}$-quasiconvex functions says that for a.e. $x\in Q$
		\begin{align*}
			\tilde L(\baf,\bai)[\,(\nabla s,c),(\nabla s,c)\,]\leq\big\langle\nu_x,\tilde L(\baf,\bai)\,\big[(\lambda_F,\lambda_\eta),(\lambda_F,\lambda_\eta)\big]\big\rangle.
		\end{align*}
		Adding $C_{pen}|s|^2$ on both sides and integrating over $Q$, we infer that 
		\begin{align*}
			\int_QC_{pen}|s|^2+&\tilde L(\baf,\bai)[\,(\nabla s,c),(\nabla s,c)\,]\\
			&\leq \int_QC_{pen}|s|^2+\big\langle\nu_x,\tilde L(\baf,\bai)\,\big[(\lambda_F,\lambda_\eta),(\lambda_F,\lambda_\eta)\big]\big\rangle\overset{\eqref{eq:step4}}\leq 0.
		\end{align*}
		However, by Proposition \ref{prop:Daf}, since $\nabla s\in W^{1,p}(Q)$ and $c\in L^q(Q)$, we know that 
		\begin{align*}
			\int_{Q}\tilde L(\bar{F}(x),\bai(x))\big[(\nabla s,c),(\nabla s,c)\big]+C_{pen}|s|^2\geq  c\int_{Q}|\nabla s|^2+|c|^2.
		\end{align*}
		In particular, combining with the above we see that $\nabla s=0$ and $c=0$, and by the Poincar\'e inequality $s=0$. We may thus apply Proposition \ref{prop: lgrd} for the sequences $(\baf_k,\bai_k)_k,\,(\al_ks_k)_k,\,(\al_kc_k)_k$ and $(\al_k)_k$ since
		\begin{align*}
			\al_k^{-2}|V_p(\al_ks_k)|^2=|s_k|^2+\al_k^{p-2}|s_k|^p&\lesssim|s_k|^2+\Lambda^F|M_k^Fs_k|^p\to 0\,\,\,\text{in}\,\,\,L^1(Q),\\
			\sup_k\int_Q\al_k^{-2}|V_p(\al_k\nabla s_k)|^2&\lesssim\sup_k\int_Q|s_k|^2+\Lambda^F\sup_k|M_k^Fs_k|^p<\infty,\\
			\sup_k\int_Q\al_k^{-2}|V_q(\al_k c_k)|^2&\lesssim\sup_k\int_Q|c_k|^2+\Lambda^\eta\sup_k|M_k^\eta c_k|^q<\infty,
		\end{align*}
		to infer that
		\begin{align*}
			0 < \frac{c_0}4 & = \liminf_{k\to\infty} \frac{\tilde c}4 \int_Q |\nabla s_k|^2+|c_k|^2 \\
			& \leq \liminf_{k\to\infty} \frac{\tilde c}{4} \int_Q |\nabla s_k|^2+|c_k|^2 + \alpha_k^{p-2}|\nabla s_k|^p+\alpha_k^{q-2}|c_k|^q\\
			& \leq \liminf_{k\to\infty}\alpha_k^{-2}\int_Q\tilde e(\baf_k+\al_k\nabla s_k,\bai_k+\al_k c_k|\baf_k,\bai_k)\\
			& = \liminf_{k\to\infty}\alpha_k^{-2}\int_Q\tilde e(\baf_k+\nabla \phi_k,\bai_k+\psi_k|\baf_k,\bai_k) + \frac{C_{pen}}{2}\int_Q|s_k|^2\overset{\eqref{eq:cinq}}\leq 0.
		\end{align*}
		But $c_0>0$, concluding the proof.
	\end{proof}
	We are now ready to prove our main result, the G\aa rding inequality. Note here that it suffices to show that the G\aa rding inequality of Theorem \ref{thm:main} remains valid for test functions $\psi$ such that $\|\psi\|_{L^q(Q)}>\eps_0$. 
	\begin{Theorem}\label{theorem:3Garding}
		There exists constants $ C_0= C_0(e,K)>0$, $ C_1= C_1(e,K)>0$ such that for all $(\baf,\bai)\in \mathcal{U}_K$, $\psi\in L^q(Q)$ with $\int_Q\psi=0$ and $\phi\in W^{1,p}_0(Q)$  it holds that
		\begin{align}\label{eq:mmgarding}
			\int_Q&\Big(|V_p(\nabla\phi(x))|^2+|V_q(\psi(x))|^2\Big)dx \\
			&\leq C_0\int_Q e(\baf(x)+\nabla\phi(x),\bai(x)+\psi(x)|\baf(x),\bai(x))dx+  C_1\int_Q|V_p(\phi(x))|^2dx.
		\end{align}
	\end{Theorem}
	\begin{proof}
		We claim that for all $\varepsilon>0$ and all $(\baf,\bai)\in \mathcal{U}_K$, $\psi\in L^q(Q)$ zero-average and $\phi\in W^{1,p}_0(Q)$ with $\|\phi\|_{L^{p}(Q)}\geq \varepsilon$ it holds that 
		\begin{align*}
			\int_Q\Big(|V_p(\nabla\phi)|^2+|V_q(\psi)|^2\Big)
			\leq C_0(\epsilon)\int_Q e(\baf+\nabla\phi,\bai+\psi|\baf,\bai)+  C_1(\epsilon)\int_Q|V_p(\phi)|^2,
		\end{align*}
		where $C_0$ and $C_1$ also depend on $\epsilon$.  By the assumed coercivity of $e$, its smoothness and the fact that $(\baf,\bai)\in \mathcal{U}_K$, we use Young's inequality to estimate
		\begin{align*}
			e(\baf+\nabla\phi,\bai+\psi|\baf,\bai) &\geq c\left(-1 + |\baf+\nabla\phi|^p+|\bai+\psi|^q\right)\\
			&- C(\delta)|e_1(\baf,\bai)|^{p'} - \delta |\nabla\phi|^p- C(\delta)|e_2(\baf,\bai)|^{q'} - \delta |\psi|^q \\
			&\geq -C(\delta)+(2^{1-p}c-\delta)|\nabla\phi|^p+(2^{1-q}c-\delta)|\psi|^q.
		\end{align*}
		Since $p\geq q$, choose $\delta=2^{-p}c$ to find that 
		\begin{align}\label{eq:1.1}
			e(\baf+\nabla\phi,\bai+\psi|\baf,\bai)\geq -C+2^{-p}c|\nabla\phi|^p+2^{-q}c|\psi|^q.
		\end{align}
		Note that since $\|\phi\|_{L^{p}(Q)}\geq \eps$, it follows that
		\[
		C \leq \frac{C}{\eps^p}\|\phi\|_{L^{p}(Q)}^p,
		\]
		so that, integrating \eqref{eq:1.1} over $Q$ with $|Q|=1$, we infer that
		\begin{equation}\label{eq:1.2}
			2^{-p}c\int_Q |\nabla\phi|^p + 2^{-q}c\int_Q|\psi|^q \leq \int_Q e(\baf+\nabla\phi,\bai+\psi|\baf,\bai) + \frac{C}{\eps^p}\int_Q|\phi|^p.
		\end{equation}
		But since $1 \leq \frac{1}{\eps^p}\|\phi\|_{L^{p}(Q)}^p\lesssim \frac{1}{\eps^p}\|\nabla\phi\|_{L^{p}(Q)}^p$ we have that
		\begin{align*}
			\int_Q|V_p(\nabla\phi)|^2=\|\nabla\phi\|_{L^{2}(Q)}^2+\|\nabla\phi\|_{L^{p}(Q)}^p&\leq 1+2\|\nabla\phi\|_{L^{p}(Q)}^p\lesssim \left( \frac{1}{\eps^p}+2\right)\|\nabla\phi\|_{L^{p}(Q)}^p,
		\end{align*}
		and similarly 
		\begin{align*}
			\int_Q|V_q(\psi)|^2\lesssim \frac{1}{\eps^p}\|\nabla\phi\|_{L^{p}(Q)}^p+\|\psi\|_{L^q(Q)}^q.
		\end{align*}
		So combining the above two results,
		\begin{align*}
			\int_Q|V_p(\nabla\phi)|^2+\int_Q|V_q(\psi)|^2&\lesssim \left( \frac{1}{\eps^p}+1\right)\|\nabla\phi\|_{L^{p}(Q)}^p+\|\psi\|_{L^q(Q)}^q\\
			&\leq C(\eps)\left( \|\nabla\phi\|_{L^{p}(Q)}^p+\|\psi\|_{L^q(Q)}^q\right).
		\end{align*}
		Together with \eqref{eq:1.2} we infer that
		\begin{align*}
			C^{-1}(\eps) \int_Q|V_p(\nabla\phi)|^2+|V_q(\psi)|^2 \leq\int_Q e(\baf+\nabla\phi,\bai+\psi|\baf,\bai) + \frac{C}{\eps^p}\int_Q|V_p(\phi)|^2,
		\end{align*}
		which is the desired inequality. However, Theorem \ref{thm:main} says that there exists $\eps_0>0$ such that whenever $\|\phi\|_{L^p(Q)}\leq\eps_0$ it holds that 
		\begin{align*}
			\int_Q|V_p(\nabla\phi)|^2+|V_q(\psi)|^2 \leq\tilde C_0\int_Q e(\baf+\nabla\phi,\bai+\psi|\baf,\bai) + \tilde C_1\int_Q|V_p(\phi)|^2.
		\end{align*}
		Choosing $\eps=\eps_0$ we conclude the proof.
	\end{proof}

	\newpage
	\section*{Appendix A. Proof of Lemma \ref{growths}}
	
	\begin{proof}
		We define the associated Hessian $L_f$ as follows
		\begin{align*}
			L_f(\lambda_1,\lambda_2)[(\xi_1,\xi_2),(\xi_1,\xi_2)]:=&f_{FF}(\lambda_1,\lambda_2)\,\xi_1:\xi_1+f_{\eta\eta}(\lambda_1,\lambda_2)\,\xi_2\,\xi_2\\
			&+ 2f_{F\eta}(\lambda_1,\lambda_2)\,\xi_1\,\xi_2.
		\end{align*}
		In the sequel we write simply $L$ instead of $L_f$.
		
		\vspace{0.2cm}
		For (a), note that if $|\xi_1|+|\xi_2|+|z_1|+|z_2|\leq 1$, it holds that
		\begin{align*}
			&\begin{aligned}
				f(\lambda_1+\xi_1,\lambda_2+\xi_2|\lambda_1,\lambda_2) - f(\lambda_1+z_1,\lambda_2+z_2|\lambda_1,\lambda_2)
			\end{aligned}\\
			&\begin{aligned}
				= \int_0^1(1-t)&\Big(\,L(\lambda_1+t\xi_1,\lambda_2+t\xi_2)[\,(\xi_1,\xi_2),(\xi_1,\xi_2)\,]\\
				&-L(\lambda_1+tz_1,\lambda_2+tz_2)[\,(z_1,z_2),(z_1,z_2)\,]\,\Big)dt
			\end{aligned}\\
			&\begin{aligned}
				=\int_0^1(1-t)&\Big(\,L(\lambda_1+t\xi_1,\lambda_2+t\xi_2)[\,(\xi_1,\xi_2),(\xi_1,\xi_2)\,]\\
				&-L(\lambda_1+t\xi_1,\lambda_2+t\xi_2)[\,(\xi_1,\xi_2),(z_1,z_2)\,]\Big)dt
			\end{aligned}\\
			&\begin{aligned}
				+\int_0^1(1-t)&\Big(\,L(\lambda_1+t\xi_1,\lambda_2+t\xi_2)[\,(\xi_1,\xi_2),(z_1,z_2)\,]\\
				&-L(\lambda_1+t\xi_1,\lambda_2+t\xi_2)[\,(z_1,z_2),(z_1,z_2)\Big)dt
			\end{aligned}\\
			&\begin{aligned}
				+\int_0^1(1-t)&\Big(\,L(\lambda_1+t\xi_1,\lambda_2+t\xi_2)[\,(z_1,z_2),(z_1,z_2)\,]\\
				&-L(\lambda_1+tz_1,\lambda_2+tz_2)[\,(z_1,z_2),(z_1,z_2)\,]\,\Big)dt
			\end{aligned}\\
			&\begin{aligned}
				=:\int_0^1(1-t)(I_1+I_2+I_3)dt.
			\end{aligned}
		\end{align*}
		Concerning the term $I$ we observe that
		\begin{align*}
			&I_1\leq|I_1|\leq|f_{FF}(\lambda_1+t\xi_1,\lambda_2+t\xi_2)||\xi_1||\xi_1-z_1|\\
			& +|f_{\eta\eta}(\lambda_1+t\xi_1,\lambda_2+t\xi_2)|\xi_2||\xi_2-z_2|
			+2|f_{F\eta}(\lambda_1+t\xi_1,\lambda_2+t\xi_2)||\xi_1\xi_2-z_1z_2|,
		\end{align*}
		and since $\xi_1\xi_2-z_1z_2=\xi_1\xi_2-z_1\xi_2+z_1\xi_2-z_1z_2$ we infer that
		\begin{align*}
			I_1&\leq C\big(|\xi_1||\xi_1-z_1|+|\xi_2||\xi_1-z_1|+|z_1||\xi_2-z_2|+|\xi_2||\xi_2-z_2|\big)\\
			&=C\big(|\xi_1|+|\xi_2|\big)|\xi_1-z_1|+C\big(|\xi_1|+|\xi_2|\big)|\xi_2-z_2|.
		\end{align*}
		Similarly for $I_2$ we find that 
		\begin{align*}
			I_2&\leq C\big(|\xi_1||\xi_1-z_1|+|\xi_2||\xi_1-z_1|+|z_1||\xi_2-z_2|+|\xi_2||\xi_2-z_2|\big)\\
			&=C\big(|z_1|+|z_2|\big)|\xi_1-z_1|+C\big(|\xi_1|+|z_2|\big)|\xi_2-z_2|.
		\end{align*}
		Considering the third term since for $|z_1|+|z_2|\leq 1$ we have that $\big(|z_1|+|z_2|\big)^2\leq |z_1|+|z_2|$ it holds that
		\begin{align*}
			I_3\leq |I_3|&\leq C\big(|\xi_1-z_1|+|\xi_2-z_2|\big)\big(|z_1|+|z_2|\big)^2\\
			&\leq C\big(|\xi_1-z_1|+|\xi_2-z_2|\big)\big(|z_1|+|z_2|\big).
		\end{align*}
		Combining the above we deduce that
		\begin{align*}
			&|f(\lambda_1+\xi_1,\lambda_2+\xi_2|\lambda_1,\lambda_2) - f(\lambda_1+z_1,\lambda_2+z_2|\lambda_1,\lambda_2)|\\
			&\lesssim \big(|\xi_1|+|\xi_2|+|z_1|+|z_2|\big)|\xi_1-z_1|+\big(|\xi_1|+|\xi_2|+|z_1|+|z_2|\big)|\xi_2-z_2|.
		\end{align*}
		Now, for $|\xi_1|+|\xi_2|+|z_1|+|z_2|> 1$ 
		\begin{align*}
			&|f(\lambda_1+\xi_1,\lambda_2+\xi_2|\lambda_1,\lambda_2) - f(\lambda_1+z_1,\lambda_2+z_2|\lambda_1,\lambda_2)|\\
			&=|f(\lambda_1+\xi_1,\lambda_2+\xi_2) - f(\lambda_1+z_1,\lambda_2+z_2)+f_F(\lambda_1,\lambda_2)(\xi_1-z_1)\\
			&+f_\eta(\lambda_1,\lambda_2)(\xi_2-z_2)|\\
			&\leq|f(\lambda_1+\xi_1,\lambda_2+\xi_2) - f(\lambda_1+z_1,\lambda_2+\xi_2)|\\
			&+|f(\lambda_1+z_1,\lambda_2+\xi_2) - f(\lambda_1+z_1,\lambda_2+z_2)|\\
			&+C|\xi_1-z_1|+C|\xi_2-z_2|=:J_1+J_2+J_3
		\end{align*}
		So, from the growth assumptions on $f$ and on its partial derivatives we infer that 
		\begin{align*}
			J_1&\leq \int_0^1|f_F\big(\lambda_1+z_1+t(\xi_1-z_1),\lambda_2+\xi_2\big)|dt\cdot|\xi_1-z_1|\\
			&\lesssim (1+|z_1|^{p-1}+|\xi_1|^{p-1}+|\xi_2|^{q\frac{p-1}{p}})|\xi_1-z_1|\\
			&\lesssim (|\xi_1|+|\xi_2|+|z_1|+|z_2|+|z_1|^{p-1}+|\xi_1|^{p-1}+|\xi_2|^{q\frac{p-1}{p}})|\xi_1-z_1|.
		\end{align*}
		Similarly for the second term 
		\begin{align*}
			J_2\lesssim \big(|\xi_1|+|\xi_2|+|z_1|+|z_2|+|\xi_2|^{q-1}+|z_2|^{q-1}+|z_1|^{p\frac{q-1}{q}}\big)|\xi_2-z_2|.
		\end{align*}
		For $J_3$, since $|\xi_1|+|\xi_2|+|z_1|+|z_2|> 1$ we have that
		\begin{align*}
			J_3\lesssim \big(|\xi_1|+|\xi_2|+|z_1|+|z_2|\big)\,|\xi_1-z_1|+\big(|\xi_1|+|\xi_2|+|z_1|+|z_2|\big)\,|\xi_2-z_2|,
		\end{align*}
		and together with the first two terms we deduce that 
		\begin{align*}
			&|f(\lambda_1+\xi_1,\lambda_2+\xi_2|\lambda_1,\lambda_2) - f(\lambda_1+z_1,\lambda_2+z_2|\lambda_1,\lambda_2)| \\
			&\leq C (|\xi_1|+|\xi_2|+|z_1| + |z_2| +|\xi_1|^{p-1}+ |z_1|^{p-1}+|\xi_2|^{q\frac{p-1}{p}})|\xi_1-z_1|\\
			&+C(|\xi_1|+|\xi_2|+|z_1| + |z_2| +|\xi_2|^{q-1}+ |z_2|^{q-1}+|z_1|^{p\frac{q-1}{q}})|\xi_2-z_2|.
		\end{align*}
		For the second part we just set $(z_1,z_2)=(0,0)$ in the above inequality and apply Young's inequality to conclude the proof of (a).
		\par 
		Concerning (b) we see that for $|\xi_1|+|\xi_2|\leq 1$ and $\xi=(\xi_1,\xi_2)$, it holds
		
		\begin{align*}
			&|f(\lambda_1+\xi_1,\lambda_2+\xi_2|\lambda_1,\lambda_2) - f(\mu_1+\xi_1,\mu_2+\xi_2|\mu_1,\mu_2)|\\
			&\leq \int_0^1\Big|\,L(\lambda_1+t\xi_1,\lambda_2+t\xi_2)[\xi,\xi]-L(\mu_1+t\xi_1,\mu_2+t\xi_2)[\xi,\xi]\,\Big|dt\\
			&\leq C\big(|\lambda_1-\mu_1|+|\lambda_2-\mu_2|\big)\big(|\xi_1|^2+|\xi_2|^2\big).
		\end{align*}
		When $|\xi_1|+|\xi_2|> 1$, from the growth of $Df=(f_F,f_\eta)$, we have that
		\begin{align*}
			&|f(\lambda_1+\xi_1,\lambda_2+\xi_2|\lambda_1,\lambda_2) - f(\mu_1+\xi_1,\mu_2+\xi_2|\mu_1,\mu_2)|\\
			&\leq\big|f(\lambda_1+\xi_1,\lambda_2+\xi_2)-f(\mu_1+\xi_1,\mu_2+\xi_2)\big|+\big|f(\lambda_1,\lambda_2)-f(\mu_1,\mu_2)\big|\\
			&+\big|Df(\lambda_1,\lambda_2)-Df(\mu_1,\mu_2)\big|\cdot|\xi|\\
			&\leq\big|f(\lambda_1+\xi_1,\lambda_2+\xi_2)-f(\mu_1+\xi_1,\lambda_2+\xi_2)\big|\\
			&+\big|f(\mu_1+\xi_1,\lambda_2+\xi_2)-f(\mu_1+\xi_1,\mu_2+\xi_2)\big|+\big(|\xi_1|+|\xi_2|\big)\big(|\lambda_1-\mu_1|+|\lambda_2-\mu_2|\big)\\
			&\lesssim \int_0^1|f_F\big(\xi_1+\mu_1+t(\lambda_1-\mu_1),\lambda_2+\xi_2\big)|dt\cdot|\lambda_1-\mu_1|\\
			&+\int_0^1|f_\eta\big(\mu_1+\xi_1,\xi_2+\mu_2+t(\lambda_2-\mu_2)\big)|dt\cdot|\lambda_2-\mu_2|\\
			&+\big(|\xi_1|+|\xi_2|\big)\big(|\lambda_1-\mu_1|+|\lambda_2-\mu_2|\big)\\
			&\lesssim \left(1+|\xi_1|^{p-1}+|\xi_2|^{q\frac{p-1}{p}}\right)|\lambda_1-\mu_1|+\left(1+|\xi_1|^{p\frac{q-1}{q}}+|\xi_2|^{q-1}\right)|\lambda_2-\mu_2|\\
			&+\big(|\xi_1|+|\xi_2|\big)\big(|\lambda_1-\mu_1|+|\lambda_2-\mu_2|\big).
		\end{align*}
		But since $|\xi_1|+|\xi_2|> 1$ we have that   $|\xi_1|^{p-1}+|\xi_2|^{q\frac{p-1}{p}},\,|\xi_1|^{p\frac{q-1}{q}}+|\xi_2|^{q-1}\leq 2+|\xi_1|^{p}+|\xi_2|^{q}\lesssim |\xi_1|+|\xi_2|+|\xi_1|^{p}+|\xi_2|^{q}\lesssim |\xi_1|^2+|\xi_2|^2+|\xi_1|^{p}+|\xi_2|^{q}$. Combining both cases we infer that
		\begin{align*}
			|f(\lambda_1+\xi_1,\lambda_2+\xi_2|\lambda_1,\lambda_2)& - f(\mu_1+\xi_1,\mu_2+\xi_2|\mu_1,\mu_2)|\\
			&\leq C \big(|\lambda_1-\mu_1|+|\lambda_2-\mu_2|\big)\big(|\xi_1|^2+|\xi_2|^2+|\xi_1|^{p}+|\xi_2|^{q}\big),
		\end{align*}
		and by choosing $R<\delta/C$ we conclude the proof of part (b).
		\par 
		Regarding the proof of (c), again for $|\xi_1|+|\xi_2|\leq 1$, since $p,q\geq 2$
		\begin{align*}
			f(\lambda_1+\xi_1,\lambda_2+\xi_2|\lambda_1,&\lambda_2)=\int_0^1(1-t)L(\lambda_1+t\xi_1,\lambda_2+t\xi_2)[\xi,\xi]dt\geq -\tilde d_2|\xi|^2\\
			&\geq \tilde d_1\big[(|\xi_1|^p-|\xi_1|^2)+(|\xi_2|^q-|\xi_2|^2)\big]-\tilde d_2(|\xi_1|^2+|\xi_2|^2).
		\end{align*}
		For the case $|\xi_1|+|\xi_2|> 1$, from the coercivity of $f$ we have that
		\begin{align*}
			&f(\lambda_1+\xi_1,\lambda_2+\xi_2|\lambda_1,\lambda_2)=f(\lambda_1+\xi_1,\lambda_2+\xi_2)-f(\lambda_1,\lambda_2)-f_F(\lambda_1,\lambda_2)\xi_1\\
			&-f_\eta(\lambda_1,\lambda_2)\xi_2\\
			&\geq \tilde d_3\big(-1+|\xi_1|^p+|\xi_2|^q\big)-\tilde d_4\big(|\xi_1|+|\xi_2|\big)\\
			&\geq \tilde d_3\big(|\xi_1|^p+|\xi_2|^q\big)-\tilde d_4\big(|\xi_1|^2+|\xi_2|^2\big).
		\end{align*}
		So, combining the two cases, we may choose $d_1,d_2>0$ to conclude the proof of the lemma.
	\end{proof}
	
	\section*{Appendix B. Symmetrisable Systems of Conservation Laws}
	
	Systems of conservation laws describing the evolution of  a function $U :  \mathbb{R}^+ \times \mathbb{R}^d  \to \mathbb{R}^n$ 
	have the form
	\begin{align}
		\label{syscl}
		\partial_t A(U) +\partial_{\alpha}f_{\alpha}(U)=0
	\end{align}
	where $A$ and $ f_\alpha : \mathcal{O} \subset  \mathbb{R}^n \to \mathbb{R}^n$, $\alpha =1, \dots, d$,  are smooth functions 
	describing fluxes. Here the matrix $A(U)$ is globally  invertible on the domain of definition $\mathcal{O} \ni U$ and $\nabla A(U)$ is 
	nonsingular.  System \eqref{syscl} is endowed with an entropy - entropy flux pair $\eta, q{_\alpha} :  \mathbb{R}^n \to \mathbb{R}$ if any 
	smooth solution $U(t,x) \in C^1(\mathbb{R}^n)$ of \eqref{syscl} satisfies the additional conservation law of entropy
	\begin{align}
		\label{entrpair}
		\partial_t \eta(U) +\partial_{\alpha}q_{\alpha}(U) = 0 \, .
	\end{align}
	This is equivalent to the existence of a multiplier $G : \mathbb{R}^n \to \mathbb{R}^n$ which is a smooth function of the 
	solution $G = G(U)$ satisfying the relations
	\begin{equation}
		\label{entropyG}
		\begin{aligned}
			G \cdot \nabla A &= \nabla \eta 
			\\
			G \cdot \nabla f_\alpha  &= \nabla q_\alpha
		\end{aligned}
	\end{equation}
	or equivalently the relations
	\begin{equation}
		\label{entropystruc}
		\begin{aligned}
			\nabla G^T  \nabla A  &= \nabla A^T \nabla G
			\\
			\nabla G^T  \nabla f_\alpha  &= {\nabla f_\alpha}^T \nabla G\,.
		\end{aligned}
	\end{equation}
	In particular, whenever \eqref{entropyG} is satisfied by a smooth solution of \eqref{syscl}, then it also satisfies the entropy
	production \eqref{entrpair}.
	
	Suppose now that \eqref{syscl} is endowed with a smooth entropy pair $\eta - q_{\alpha}$, that is for some multiplier $G(U)$ 
	relations \eqref{entropystruc} are satisfied. We can rewrite \eqref{syscl} for smooth solutions, in the form of an equivalent system with
	symmetric coefficients:
	\begin{equation}
		(\nabla G^T  \nabla A ) \partial_t U + ( \nabla G^T  \nabla f_\alpha ) \partial_\alpha U = 0.
	\end{equation}
	The hypothesis
	$$
	\nabla G^T  \nabla A  > 0
	$$
	guarantees that the system \eqref{syscl} is symmetrisable in the sense of Friedrichs and Lax~\cite{MR285799}, it has real eigenvalues and it is hyperbolic. Moreover, it induces a relative 
	entropy identity and therefore a notion of stability for the system \eqref{syscl}, see \cite{MR285799,christoforou2016relative}. 
	Using \eqref{entrpair},  it can be equivalently expressed in the form
	\begin{equation}
		\label{symm}
		\nabla^2 \eta - \sum_{k=1} G^k  \nabla^2 A^k   > 0 \, .
	\end{equation}
	For weak solutions the entropy pair $\eta - q_{\alpha}$ gives rise to a notion of admissibility. The function $U \in L^1_{loc}(\mathbb{R}^n)$ is an entropy weak solution if it satisfies, in the sense of distributions, \eqref{syscl} and the entropy inequality
	\begin{align}
		\label{entrsoln}
		\partial_t \eta(U) +\partial_{\alpha}q_{\alpha}(U) \le 0.
	\end{align}
	
	Adiabatic thermoelasticity \eqref{e:adiabatic_thermoelasticity} fits into the general form of system \eqref{syscl},  by setting
	$$
	U = (F, v, \eta) \quad A(U) = \big ( F, v, \tfrac{1}{2} |v|^2 + e (F, \eta) \big ).
	$$
	The positivity condition $ \theta = \partial  e/\partial  \eta > 0$ for the temperature guarantees that $A(U)$ is invertible and 
	$\nabla A(U)$ is nonsingular. By construction of the theory, there is a multiplier $G(U)$ that leads to the entropy pair $\check \eta (U)$~-~$\check q_\alpha (U)$ with
	$$
	\begin{aligned}
		\check \eta (U) :=   - \eta,  \quad  \check q_\alpha (U) := 0, \quad
		G(U) = \tfrac{1}{\theta(F,\eta)} \bigg ( \frac{\partial  e}{\partial  F}(F,\eta)  , v , -1 \bigg ) \, .
	\end{aligned}
	$$
	We may compute that 
	\begin{equation}
		\label{hyperbolic(F,eta)}
		\nabla^2 \check \eta (U)  - \sum_{k=1} G^k(U)  \nabla^2 A^k(U)
		=
		\frac{1}{e_{\eta}}
		\begin{pmatrix}
			e_{FF} & 0 & e_{F \eta}
			\\
			0 &  1 & 0
			\\
			e_{F \eta} & 0 & e_{\eta \eta}
		\end{pmatrix}
	\end{equation}
	and thus the condition of symmetrisability \eqref{symm} amounts to $e(F, \eta)$ strongly convex and $\theta(F,\eta)=\frac{\partial  e (F,\eta)}{\partial  \eta} > 0$.
	Convexity of $e(F,\eta)$ suffices to apply the standard 
	theory of conservation laws to \eqref{e:adiabatic_thermoelasticity} and, in that case, the entropy admissibility inequality \eqref{entrsoln} amounts to the growth of the physical entropy.

		\section*{Acknowledgments}
		KK and AV acknowledge the support of the Dr Perry James (Jim) Browne Research Centre   
		on Mathematics and its Applications of the University of Sussex.
		The article was partially written while MG was a PhD student at King Abdullah University of Science and Technology (KAUST), Saudi Arabia.

\end{document}